\documentclass[12pt]{article}

\usepackage{amssymb}

\usepackage{authblk, setspace, subfig, caption}
\usepackage{url}
\usepackage{amsmath}
\usepackage{mathtools}

\usepackage{amssymb,amscd} 
\usepackage{enumerate} 
\usepackage[active]{srcltx} 
\usepackage{amsthm, leftidx}

\usepackage{multirow, booktabs}
\usepackage{algorithm}
\usepackage{algpseudocode}

\usepackage{graphicx}

\usepackage{authblk}

\usepackage{tikz}
\usetikzlibrary{mindmap,shadows, calc}
\usepackage[hidelinks,pdfencoding=auto]{hyperref}

\renewcommand{\tilde} [1] {{{#1}^{\mathrm{r}}}}

\theoremstyle{definition}
\newtheorem{defn}{Definition}
\newtheorem{rem}{Remark}
\newtheorem{ex}{Example}

\theoremstyle{plain}
\newtheorem{thm}{Theorem}
\newtheorem{lem}{Lemma}
\newtheorem{prop}{Proposition}
\newtheorem{cor}{Corollary}

\newcommand{\bRR}{\overline{\mathbb{R}}}
\newcommand{\RR}{\mathbb{R}}

\newcommand{\marphi}{\phi}
\newcommand{\marpsi}{\psi}
\newcommand{\marcop}{C^{\mathrm{M}}}
\newcommand{\marsetcop}{{\mathcal C^{\mathrm{M}}}}

\newcommand{\mmphi}{\phi}
\newcommand{\mmpsi}{\chi}
\newcommand{\mmcop}{{C^{\mathrm{MM}}}}
\newcommand{\mmsetcop}{{\mathcal C^{\mathrm{MM}}}}

\def\lpr{\underline P}
\def\upr{\overline P}

\usepackage[square,sort,comma,numbers]{natbib} 
\usepackage{algorithm}
\usepackage{algpseudocode}

\usepackage{hyperref}

\newcommand{\low}[1]{{\underline{#1}}}
\newcommand{\up}[1]{{\overline{#1}}}

\begin{document}

\title{Constructing copulas from shock models with imprecise distributions}
\author{Matja\v{z} Omladi\v{c}\\Institute of Mathematics, Physics and Mechanics, Ljubljana, Slovenia\\matjaz@omladic.net
\and Damjan \v{S}kulj\\Faculty of Social Sciences, University of Ljubljana, Slovenia\\damjan.skulj@fdv.uni-lj.si}

\maketitle

\begin{abstract}
The omnipotence of copulas when modeling dependence given marg\-inal distributions in a multivariate stochastic situation is assured by the Sklar's theorem. Montes et al.\ (2015) suggest the notion of what they call an \emph{imprecise copula} that brings some of its power in bivariate case to the imprecise setting. When there is imprecision about the marginals, one can model the available information by means of $p$-boxes, that are pairs of ordered distribution functions. By analogy they introduce pairs of bivariate functions satisfying certain conditions. In this paper we introduce the imprecise versions of some classes of copulas emerging from shock models that are important in applications. The so obtained pairs of functions are not only imprecise copulas but satisfy an even stronger condition. The fact that this condition really is stronger is shown in Omladi\v{c} and Stopar (2019) thus raising the importance of our results. The main technical difficulty in developing our imprecise copulas lies in introducing an appropriate stochastic order on these bivariate objects.
\end{abstract}


\section{Introduction}
In this paper we propose copulas arising from shock models in the presence of probabilistic uncertainty, which means that probability distributions are not necessarily precisely known. Copulas have been introduced in the precise setting by A.\ Sklar \cite{Skla}, who considered copulas as functions satisfying certain conditions -- in bivariate case they are functions of two variables $C(u,v)$. They can be defined equivalently as joint distribution functions of random vectors with uniform marginals. He proved a two-way theorem: Firstly, given random variables $X$ and $Y$ with respective marginal distributions $F$ and $G$ and a copula $C$, the function $C(F(x),G(y))$ is a joint distribution of a random vector $(X,Y)$ having distributions $F$ and $G$ as its marginals. Secondly, given a random vector $(X,Y)$ with joint distribution $H(x,y)$ there exists a copula $C(u,v)$ such that $H(x,y)=C(F(x),G(x))$, where $F$ and $G$ are the marginal distribution functions of the respective random variables $X$ and $Y$.

There are various reasons for imprecision, such as scarcity of available information, costs connected to acquiring precise inputs or even inherent uncertainty related to phenomena under consideration. Ignoring imprecision may lead to deceptive conclusions and consequentially to harmful decisions, especially if the conclusions are backed by seemingly precise outputs. The theories of imprecise probabilities that have been developed in recent decades aim at providing methods whose results would faithfully reflect the imprecision of input information. The probabilistic imprecision is quite often described with sets of possible probability distributions, consistent with the available information, instead of a single precise distribution. The sets are represented by various types of constraints, ranging from the most general lower and upper previsions to more specific lower and upper probabilities, $p$-boxes, belief and possibility functions, and other models.

In recent years, methods of imprecise probabilities \cite{augustin2014introduction} have been applied to various areas of probabilistic modelling, such as stochastic processes \cite{decooman-2008-a, skulj:09IJAR}, game theory \cite{MirandaMontes18, Nau2011}, reliability theory \cite{Coolen2004, Oberguggenberger2009, utkin2007, Yu2016}, decision theory \cite{Jansen2018, Montes2014, Troffaes2007}, financial risk theory \cite{Pelessoni2003, Vicig2008} and others. However, it is only recently that imprecise models involving copulas have been proposed.
\citet{montes2015} introduce the concept of an imprecise copula and connect it to the theory of bivariate $p$-boxes \cite{pelessoni2016} (an extension of univariate $p$-boxes \cite{Ferson2003,Troffaes2001}). A version of Sklar's theorem, actually the first part of it (in the order explained in the first paragraph of this introduction) is stated there. 
It follows from their results that the Sklar's theorem is not fully valid in their approach to the imprecise case and in this paper we give some more evidence of this fact. Based on their work we propose an imprecise version of two important families of copulas induced by shock models. So, the reader is assumed familiar with the work \cite{montes2015} and general theory of imprecise probabilities. However, only recently, a new approach appeared \cite{OmSt2} that produces a more general Sklar's type theorem. The distinction between the two approaches is fundamental; nevertheless it does not effect our work.

Another application of the theory of copulas to models of imprecise probabilities has been proposed by Schmelzer~\cite{Schmelzer2015, Schmelzer2018} where copulas are used to describe the dependence between random sets. As a matter of fact, it turns out that sets of copulas are needed to describe such dependence instead of a single copula. Moreover, he also shows in \cite{Schmelzer2015b} that in the special case of minitive belief functions Sklar's theorem actually holds, which means that a single copula is sufficient to express the dependence relation.

In this paper we restrict ourselves to two families of shock-model induced copulas, Marshal's copulas and maxmin copulas, both only in the bivariate case. These copulas are caused by shock models, i.e.\ they arise naturally as models of joint distributions for random variables representing lifetimes of components affected by shocks. Two types of shocks are considered in these models, the first type only affects each one of the two components (the idiosyncratic shocks), while the second one simultaneously affects both components (the exogenous shock). In the original Marshall's case (cf.\ \cite{marshall1996} based on an earlier work of Marshall and Olkin \cite{MaOl}) both types of shocks cause the component to cease to work immediately. Recently a new family of copulas has been proposed by Omladi\v{c} and Ru\v{z}i\'{c} \cite{omladic2016} where the exogenous (i.e.\ systemic) shock has a detrimental effect on one of the components and beneficial effect on the other one.

So, in the precise probability setting one assumes two components whose lifetimes are random variables denoted by $U$ and $V$ respectively. They may be affected by three shocks whose occurrence times are denoted by $X, Y$, and $Z$. The first two shocks affect only the first and the second component respectively, while the third shock affects simultaneously both components. In the case of Marshall's copulas, each of the three shocks is fatal for the corresponding component meaning that it causes the component to stop operating. Thus, the lifetimes of both components are equal to
\begin{equation}\label{intro:eq:marshall}
  U = \min\{ X, Z \}\ \ \mbox{and}\ \ V = \min\{ Y, Z \}.
\end{equation}
The maxmin copulas arise from a similar underlying model of shocks, only that the exogenous shock affects the first component in a beneficial way and the second one in a detrimental way, so that
\begin{equation}\label{intro:eq:maxmin}
  U = \max\{ X, Z \}\ \ \mbox{and}\ \ V = \min\{ Y, Z \}.
\end{equation}
If the respective distribution functions of $X$ and $Y$ are denoted by $F$ and $G$, then by the Sklar's theorem the joint distribution function of random vector $(X,Y)$ is $C(F, G)$, where $C$ is the bivariate Marshall's copula in the first case and the bivariate maxmin copula in the second case.

The history of copulas, starting with the already mentioned paper of A.\ Sklar in 1959 \cite{Skla}, has been too rich to list here all of the references. Let us limit ourselves to more recent papers related to our investigation, including those dealing with order statistics \cite{AnKoTa,AvGeKo,JaRy,MeSa,NaSp,Schm}, those introducing new classes of copulas \cite{DuJa,JwBaDeMe14,JwBaDeMe15,RoLaUbFl}, the ones connected to non-additive measures and perturbations \cite{DuFeSaUbFl,DuMePaSe, KlLiMePa, KlMeSpSt,MeKoKo} and some of those devoted to other important subjects \cite{BeBoCiSaPlSa,DuJa,FrNe, GeNe}; the monographs in the area include \cite{DuSe,Joe, Nels}. The history of shock models induced copulas started with the seminal paper by Marshall and Olkin in 1967 \cite{MaOl}, as already mentioned, although they have not worked with copulas yet: they just gave the formulas for joint distributions of the model in the case of exponential marginals. Nevertheless, in the half of the century since then it has become the most cited paper of this theory. A.\ W.\ Marshall in 1996 was the first one to combine this idea with the theory of copulas and found the famous formula for general marginals that bares his name. A substantial theory on shock model copulas has developed in the meantime including \cite{ChDuMu,ChMu, DuKoMeSe,Hu,LiMcNe,Mu}. The third milestone on the path of shock model induced copulas was made in 2015 by F.\ Durante, S.\ Girard and G.\ Mazo \cite{DuGiMa,DuGiMa1} who describe a general  construction principle for copulas based on shock models. In \cite{DuGiMa} a new possibility in the research of shock models induced copulas was introduced, i.e. what they call asymmetric linkage functions (such as ``max'' and ``min'' in \eqref{intro:eq:maxmin}) as opposed to symmetric ones in Marshall's copulas (such as ``min'' and ``min'' in \eqref{intro:eq:marshall}). One might call these copulas ``non-Marshall'' shock model copulas. The first ones of the kind seem to have been introduced in 2016 \cite{omladic2016} followed by \cite{DuOmOrRu,KoOm, KoBuKoMoOm1,KoBuKoMoOm2}. It is perhaps somewhat surprising that the first citations in engineering papers of this kind of copulas appeared only two years after the first appearance of the model itself \cite{LiSoZh,LiZhSo}.

A comprehensive list of references of concrete applications of shock-based copulas would be too long to present here, so let us limit ourselves to four of them, relatively recent ones and in quite different fields. An application in actuarial sciences is given in \cite{LiMcNe}, in life sciences in \cite{Hu}, both on a relatively general level, while a true data application in finance and banking is given in \cite{ChMu}, and in hydrology in \cite{DuOk}. True data applications in copula theory depend also on having access to the appropriate data and means to process them.

The main contribution of this paper is a proposal of the imprecise versions of the two shock model based copulas described above. Our approach is by no means the first application of imprecise probabilities to reliability theory. For an exhaustive review on imprecise reliability, the reader is referred to \cite{utkin2007}. In \cite{Troffaes2013}, an approach to common-cause failure models with imprecise probabilities is proposed. It is based on an idea that is somewhat similar to our shock models. Thus, in addition to individual failure rates of components, common rates that cause simultaneous failure for groups of components are included. Their work then focuses on the estimation of the underlying parameters, which differs from our more theoretical approach of modelling dependence between failure times.

In order to achieve our goal, we need to introduce an order on the pairs of functions generating these copulas, denoted by $\marphi,\marpsi$ in the Marshall's case, and by $\mmphi,\mmpsi$ in the maxmin case. This tool is developed in Section 4, in Subsection 4.1 for the first pair, in Subsection 4.2 for the second one. We find a nontrivial way to define an order on the generating functions induced by the order on $p$-boxes of the endogenous shocks. The fact that the second ``linkage'' of the maxmin copulas (i.e. the ``min'') reverses the order in the $p$-box of the corresponding shock, has surprising consequences on the obtained imprecise copula (viz.\ Remark 4 at the end of Section 5). The so obtained bivariate $p$-box cannot be well represented solely by the infimum and supremum of its elements as one might expect.

It turns out that the imprecise copulas emerging from our efforts satisfy not only the definition proposed in \citet{montes2015} but also the following stronger condition: An imprecise copula is coherent if
\[
\underline{C}=\inf\{C\colon C\ \mbox{is a copula s.t.}\ \underline{C}\le C\le\overline{C}\}
\]
and
\[
\overline{C}=\sup\{C\colon C\ \mbox{is a copula s.t.}\ \underline{C}\le C\le\overline{C}\}.
\]
We refer the interested reader to Omladi\v{c} and Stopar \cite{OmSt} (cf.\ also \cite{OmSt2}) for more details on this condition and the proof that it really is stronger. Namely, they give an example of an imprecise copula $(\underline{C},\overline{C})$ according to the definition of \cite{montes2015} such that the set $\{C\colon C\ \mbox{is a copula s.t.}\ \underline{C}\le C\le\overline{C}\}$ is empty. In the light of this fact our results are gaining an additional value.

The paper is organized as follows. Section 2 brings the preliminaries on copulas and on the imprecise setting. Section 3 presents an overview of Marshall's copulas and maxmin copulas.  Section 4 develops the main tools needed in the paper and Section 5 stages the main results. We give a detailed description on the imprecise copulas obtained from the shock models with imprecise endogenous shocks for both Marshall's type of shock models (Theorem \ref{main:marshall}) and for the maxmin type ot these models (Theorem \ref{main:maxmin}).

\section{Imprecise distribution functions and copulas}
\subsection{Copulas}
Copulas present a very convenient tool for modeling dependence of random variables free of their marginal distributions -- only when one inserts these distributions into a copula, they turn into a joint distribution.
\begin{defn}\label{imprecise copula}
	A function $C\colon [0, 1]\times [0, 1]\to [0, 1]$ is called a \emph{copula} if it satisfies the following conditions:\
	\begin{enumerate}[(C1)]
		\item $C(u, 0) = C(0, v) = 0$ for every $u, v \in [0, 1]$;
		\item $C(u, 1) = u$ and $C(1, v) = v$ for every $u, v \in [0, 1]$;
		\item $C(u_2, v_2)-C(u_1, v_2)-C(u_2, v_1)+C(u_1, v_1) \geqslant 0$ for every $0\leqslant u_1\leqslant u_2\leqslant 1$ and $0\leqslant v_1\leqslant v_2\leqslant 1$.
	\end{enumerate}
\end{defn}
\begin{thm}[Sklar's theorem]
	Let $F\colon \bRR\times \bRR\to [0, 1]$, where $\bRR = \RR\cup \{-\infty, \infty\}$, be a bivariate distribution function with marginals $F_X$ and $F_Y$. Then there exists a copula $C$ such that
	\begin{equation}\label{eq-copula-sklar}
	F(x, y) = C(F_X(x), F_Y(y)) \text{ for all } (x,y)\in\overline{\mathbb{R}}\times\overline{\mathbb{R}};
	\end{equation}
	and conversely, given any copula $C$ and a pair of distribution functions $F_X$ and $F_Y$, equation \eqref{eq-copula-sklar} defines a bivariate distribution function.
\end{thm}


\subsection{Coherent lower and upper probabilities, $p$-boxes}
We first introduce briefly the basic concepts and ideas of imprecise probability models. For a detailed treatment, the reader is referred to \cite{augustin2014introduction, walley91}.
Let $\Omega$ be a possibility space, and $\mathcal A$ a collection of its subsets, called \emph{events}.
Usually we assume $\mathcal A$ to be an algebra, but not necessarily a $\sigma$-algebra.

The concept of precise probability on the measurable space $(\Omega, \mathcal A)$ can be generalised by allowing probabilities of events in $\mathcal A$ to be given in terms of intervals $[\low P(A), \up P(A)]$ rather than precise values.
The functions $\low P\leqslant \up P$ are mapping events to their lower and upper probability bounds and are respectively called \emph{lower} and \emph{upper probabilities}\footnote{Lower and upper probabilities are a special case of even more general lower and upper previsions \cite{walley91}.}.
If $\mathcal A$ is an algebra (or at least closed for complements), then the following conjugacy relation between lower and upper probabilities is usually required:
\begin{equation}\label{eq-conjugacy}
	\upr (A) = 1-\lpr(A^c) ~\text{for every } A\in \mathcal A.
\end{equation}
To every pair of lower and upper probabilities $\lpr$ and $\upr$ we can also associate the set
\begin{equation*}
	\mathcal M = \{ P \colon P \text{ is a finitely additive probability on } \mathcal A, \lpr \leqslant P \leqslant \upr \}.
\end{equation*}
It is clear from the above definition that $\lpr \le \upr$ is a necessary condition for the set $\mathcal M$ to be non-empty. 

Another central question regarding a pair of lower and upper probabilities is whether the bounds are pointwise limits of the elements in $\mathcal M$:
\begin{equation*}
	\lpr (A) = \inf_{P\in\mathcal M} P(A), \qquad \upr (A) = \sup_{P\in\mathcal M} P(A) \qquad \text{ for every} A\in \mathcal A.
\end{equation*}
If the above conditions are satisfied, $\lpr$ and $\upr$ are said to be \emph{coherent} lower and upper probabilities respectively.
In the case of coherence, the conjugacy condition \eqref{eq-conjugacy} is automatically fulfilled, which among others means that if a lower probability $\lpr$ is coherent, then it uniquely determines the corresponding upper probability.

A simple characterization of coherence in terms of the properties of $\lpr$ and $\upr$ does not seem to be known in literature. Below we give some necessary conditions. A pair $\lpr$ and $\upr$ of coherent lower and upper probabilities satisfies the following properties:
\begin{enumerate}[(i)]
	\item $0 \leqslant \lpr(A) \leqslant \upr (A) \leqslant 1$ for every $A\in \mathcal A$.
	\item If $A\subseteq B$, then $\lpr (A) \leqslant \lpr (B)$ and $\upr (A) \leqslant \upr (B)$. (Monotonicity)
	\item $\lpr (A\cup B) \geqslant \lpr(A) + \lpr (B).$ (Superadditivity)
\end{enumerate}


Instead of the full structure of probability spaces, we are often concerned only with the distribution functions of specific random variables. The set of relevant events where the probabilities have to be given then shrinks considerably. In the precise case, a single distribution function $F$ describes the distribution of a random variable $X$, which gives the probabilities of the events of the form $\{ X\leqslant x\}$. Thus $F(x) = P(X\leqslant x)$. 	Sometimes we will also consider the corresponding \emph{survival function}, which we will denote by $\hat F(x) = 1-F(x) = P(X > x)$, which is decreasing and positive.

In the imprecise case, the probabilities of the above form are replaced by the corresponding lower (and upper) probabilities, resulting in sets of distribution functions called $p$-boxes \cite{Ferson2003, Troffaes2001}. A \emph{$p$-box} is a pair $(\low F, \up F)$ of distribution functions with $\low F\leqslant \up F$, where $\low F(x) = \lpr(X\leqslant x)$ and $\up F(x) = \upr(X\leqslant x)$.  To every $p$-box we associate the set of all  distribution functions with the values between the bounds:
\begin{equation*}
	\mathcal F_{(\low F, \up F)} = \{ F \colon F \text{ is a distribution function}, \low F \leqslant F \leqslant \up F \}.
\end{equation*}
Clearly, $\mathcal F_{(\low F, \up F)}$ is a convex set of distribution functions. Conversely, since supremum and infimum of any set of distribution functions are themselves distribution functions, every set of distribution functions generates a $p$-box containing the original set.

In the theory of imprecise probabilities, precise probability denotes a probability measure that is finitely additive, and not necessarily $\sigma$-additive, as is the case in most models using classical probabilities. As far as distribution functions and $p$-boxes are concerned, this implies that a distribution function is any increasing, or more precisely non-decreasing function, mapping $\RR$ to $[0, 1]$. This is in contrast with the $\sigma$-additive case, where distribution functions are cadlag, i.e.\ continuous from the right; and their corresponding survival functions are then caglad, i.e.\ continuous from the left. No such continuity assumptions can therefore be made neither for distribution functions nor for the bounds of $p$-boxes.

\subsection{Bivariate $p$-boxes}\label{subsec:p_box}

Here we are interested in joint probability distributions that are modeled by bivariate distribution functions in the imprecise setting. Such modeling is described in \cite{montes2015,pelessoni2016}. The following basic results can be found in the latter reference.
\begin{defn}
	A map $F\colon \bRR\times \bRR \to [0, 1]$ is called \emph{standardized} if
	\begin{enumerate}[(i)]
		\item \emph{it is componentwise increasing: } $F(x_1, y) \leqslant F(x_2, y)$ and $F(x, y_1) \leqslant F(x, y_2)$ whenever $x_1\leqslant x_2$ and $y_1\leqslant y_2$ and for all $x, y\in \bRR$;
		\item $F(-\infty, y) = F(x, -\infty) = 0$ for every $x, y\in \bRR$;
		\item $F(\infty, \infty) = 1$.
	\end{enumerate}	
If in addition,
	\begin{enumerate}[(iv)]
	\item 	$F(x_2, y_2)-F(x_1, y_2)-F(x_2, y_1)+F(x_1, y_1) \geqslant 0$ for every $x_1\leqslant x_2$ and $y_1\leqslant y_2$,
	\end{enumerate}	
	then it is called a \emph{bivariate distribution function}.
	
	 A pair $(\low F, \up F)$ of standardized functions, where $\low F\leqslant \up F$, is called a \emph{bivariate $p$-box}.
\end{defn}
Notice that the bounds of a bivariate $p$-box do not need to be bivariate distribution functions themselves, as a supremum or infimum of bivariate distribution functions does not need to have this property. Further, a bivariate $p$-box is said to be \emph{coherent} if its bounds $\low F$ and $\up F$ are the lower and upper envelopes respectively of the set of bivariate distribution functions
\[
	\mathcal F_{(\low F, \up F)} = \{ F\colon \bRR \times \bRR \to [0, 1], F \text{ is a bivariate distribution function}, \low F \leqslant F \leqslant \up F \}.
\]
Note that the above set could in some cases be empty. In general, there is also no clear characterization of coherent bivariate $p$-boxes in terms of the properties of the bounds. This makes bivariate $p$-boxes substantially different from the univariate ones. Yet, this seems completely normal for imprecise probability models, where the bounds are usually not of the same type as probability distributions they bound.

For every coherent bivariate $p$-box the following conditions hold for every $x_1\leqslant x_2$ and $y_1\leqslant y_2$:
\begin{enumerate}[(i)]
	\item $\low F(x_2, y_2) + \up F(x_1, y_1) - \low F(x_2, y_1) - \low F(x_1, y_2) \geqslant 0$;
	\item $\up F(x_2, y_2) + \low F(x_1, y_1) - \low F(x_2, y_1) - \low F(x_1, y_2) \geqslant 0$;
	\item $\up F(x_2, y_2) + \up F(x_1, y_1) - \up F(x_2, y_1) - \low F(x_1, y_2) \geqslant 0$;
	\item $\up F(x_2, y_2) + \up F(x_1, y_1) - \low F(x_2, y_1) - \up F(x_1, y_2) \geqslant 0$;
\end{enumerate}
Clearly, a pair of bivariate distribution functions $\low F\leqslant \up F$ forms a coherent bivariate $p$-box, as the bounds are reached by $\low F$ and $\up F$ themselves.

\subsection{Imprecise copulas}\label{ss:imprecise}
The extension of copulas to imprecise probabilities in the way we use it in this paper is introduced by \citet{montes2015}, where also a partial generalization of Sklar's theorem is given. An imprecise copula is defined as follows.
\begin{defn}
	A pair $(\low C, \up C)$ of functions $\low C$ and $\up C$ mapping $[0, 1]\times [0, 1]$ to $[0, 1]$ is called an \emph{imprecise copula} if
	\begin{enumerate}[(i)]
		\item $\low C(u, 0) = \low C(0, v) = 0, \low C(1, v) = v, \low C(u, 1) = u$ for every $u, v\in [0, 1]$;
		\item $\up C(u, 0) = \up C(0, v) = 0, \up C(1, v) = v, \up C(u, 1) = u$ for every $u, v\in [0, 1]$;
		\item for every $0\leqslant u_1\leqslant u_2\leqslant 1$ and $0\leqslant v_1\leqslant v_2\leqslant 1$:
		\begin{enumerate}[({IC}1)]
				\item $\low C(u_2, v_2) + \up C(u_1, v_1) - \low C(u_2, v_1) - \low C(u_1, v_2) \geqslant 0$;
				\item $\up C(u_2, v_2) + \low C(u_1, v_1) - \low C(u_2, v_1) - \low C(u_1, v_2) \geqslant 0$;
				\item $\up C(u_2, v_2) + \up C(u_1, v_1) - \up C(u_2, v_1) - \low C(u_1, v_2) \geqslant 0$;
				\item $\up C(u_2, v_2) + \up C(u_1, v_1) - \low C(u_2, v_1) - \up C(u_1, v_2) \geqslant 0$;
		\end{enumerate}
	\end{enumerate}
\end{defn}
It follows from (iii) of the above definition that $\low C \leqslant \up C$, which is an important property used throughout the paper, but might not be obvious from the definition at first sight.

It follows from Definition \ref{imprecise copula} that an imprecise copula is a bivariate $p$-box with the marginals that are both uniformly distributed on the unit interval.
Imprecise copulas seem to be in a close relationship with sets of (precise) copulas. Thus, given a non-empty set of copulas $\mathcal C$, its upper and lower bounds
\begin{align}
	\low C(u, v) & = \inf_{C\in\mathcal C} C(u, v) \label{eq-copula-lower} \\
	\up C(u, v) & = \sup_{C\in \mathcal C} C(u, v) \label{eq-copula-upper}
\end{align}
form an imprecise copula. Conversely, to any imprecise copula $(\low C, \up C)$ a set of copulas
\begin{equation}\label{intermediate}
	\mathcal C = \{ C \colon C\ \mbox{copula s.t.}\ \low C \leqslant C \leqslant \up C \}.
\end{equation}
can be assigned. The authors of \cite{montes2015} propose a question, whether the lower and upper bounds of this set yield respectively $\low C$ and $\up C$ back. A  (nontrivial) counterexample to that question was given in Omladi\v{c} and Stopar \cite{OmSt}. So, we will say (just for the purpose of this paper) that an imprecise copula $(\low C,\up C)$ \emph{is coherent} if for the set $\mathcal{C}$ defined by \eqref{intermediate} relations \eqref{eq-copula-lower} and \eqref{eq-copula-upper} are satisfied. Observe that our definition of coherence on imprecise copulas is analogous to coherence of bivariate $p$-boxes defined in Subsection \ref{subsec:p_box}. We will show in the sequel that all imprecise copulas induced by shock models are automatically coherent.


\subsection{An imprecise version of the Sklar's theorem}
The situation described by Sklar's theorem includes a pair of marginal distributions $F_X$ and $F_Y$ and a copula $C$, together generating a bivariate distribution function $F = C(F_X, F_Y)$. In the imprecise case, the marginals would be replaced by $p$-boxes $(\low F_X, \up F_X)$ and $(\low F_Y, \up F_Y)$ to which an imprecise copula $(\low C, \up C)$ is applied.
\begin{thm}[Imprecise Sklar's theorem ( \cite{montes2015})]
	Let $(\low F_X, \up F_X)$ and $(\low F_Y, \up F_Y)$ be two univariate $p$-boxes and $\mathcal C$ a non-empty set of copulas, with the pointwise lower and upper bounds $(\low C, \up C)$ (see \eqref{eq-copula-lower} and \eqref{eq-copula-upper}) forming an imprecise copula. Then the functions
	\begin{equation}\label{eq-imprecise-sklar}		
		\low F = \low C(\low F_X, \low F_Y) \text{ and }\up F = \up C(\up F_X, \up F_Y)
	\end{equation}
	determine a coherent bivariate $p$-box.
\end{thm}
The converse of the above theorem does not hold in general. That is, let a bivariate $p$-box $(\low F, \up F)$ with the marginals $(\low F_X, \up F_X)$ and $(\low F_Y, \up F_Y)$ be given. There may be no imprecise copula $(\low C, \up C)$ so that \eqref{eq-imprecise-sklar} holds. Some evidence of this fact was given in \cite{montes2015} and some more will be presented in the following sections. However, recently Omladi\v{c} and Stopar \cite{OmSt2} use a slightly different approach to bivariate $p$-boxes they call restricted bivariate $p$-box that enables them to get a much more general Sklar-type theorem in the imprecise setting.

\subsection{Independent random variables}
In the case where probability distributions are known imprecisely, several distinct concepts of independence exist, such as \emph{epistemic irrelevance, epistemic independence} and \emph{strong independence} (see e.g. \cite{Couso2000,Couso2010}). However, as long as $p$-boxes are concerned, all these notions result in the \emph{factorization property}, defined as follows.

In \citet{montes2015}, the following construction is proposed. Let $p$-boxes $(\low F_X, \up F_X)$ and $(\low F_Y, \up F_Y)$ be given, corresponding to random variables $X$ and $Y$. Further, let $\lpr$ be a coherent lower probability\footnote{In the original paper, a coherent lower prevision, which is a more general model, is used.} on the product space of the domains of $X$ and $Y$, that is factorising, which means that a form of independence between $X$ and $Y$ holds. It is also shown that it does not really matter which independence concept is used if only bivariate $p$-boxes are studied, so that  we will simply say in such cases that the random variables under consideration are independent. Then $\lpr$ induces the bivariate $p$-box $(\low F, \up F)$, where $\low F(x, y)  = \low F_X(x) \low F_Y(y)$  and $\up F(x, y) = \up F_X(x) \up F_Y(y)$,
which is coherent and whose associated set of distribution functions $\mathcal F_{(\low F, \up F)}$ contains all product distribution functions $F_XF_Y$, where $\low F_X \leqslant F_X \leqslant \up F_X$ and $\low F_Y \leqslant F_Y \leqslant \up F_Y$. This construction justifies the following definition.
\begin{defn}\label{def-factorization}
	Let a pair of $p$-boxes $(\low F_X, \up F_X)$ and $(\low F_Y, \up F_Y)$ correspond to the distributions of random variables $X$ and $Y$. The bivariate $p$-box $(\low F, \up F)$ is \emph{factorizing} if
	\begin{align*}
	\low F(x, y) & = \low F_X(x) \low F_Y(y) \label{eq-factorizing-p-box-l}\\
	\up F(x, y) & = \up F_X(x) \up F_Y(y).
	\end{align*}
\end{defn}
Thus a bivariate $p$-box corresponding to the bivariate distribution of a pair of independent random variables is factorizing, regardless of the type of independence.

\section{Marshall's copulas and maxmin copulas revisited}\label{ss-mmc}
In this section we describe the two important families of copulas that are used to model the dependence between the pairs of random variables \eqref{intro:eq:marshall} and \eqref{intro:eq:maxmin} described in the introduction. Observe that this section assumes the historical setting of the shock-based copulas which means in particular that all random variables are cadlag. However, from Section \ref{sec:order} on all the cumulative distribution functions will be assumed monotone only.
\subsection{Marshall's copulas}
Copulas of the form
\begin{equation*}
	\marcop_{\marphi, \marpsi}(u, v) = \begin{cases}
	uv\min\left\{ \frac{\marphi(u)}{u}, \frac{\marpsi(v)}{v} \right\}, & uv > 0; \\
	0, & uv = 0,
	\end{cases}
\end{equation*}
where
\begin{enumerate}[(P1)]
	\item $\marphi$ and $\marpsi$ are two non-decreasing real valued maps on $[0, 1]$;
	\item $\marphi(0) = \marpsi(0) = 0$ and $\marphi(1) = \marpsi(1) = 1$; 	
	\item $\marphi^*(u) = \dfrac{\marphi(u)}{u}\colon (0, 1]\to [1, \infty]$ and $\marpsi^*(v) = \dfrac{\marpsi(v)}{v}\colon (0, 1]\to [1, \infty]$ are non-increasing,
\end{enumerate}
are called \emph{Marshall's copulas} and are well known to model the dependence between the pair of random variables \eqref{intro:eq:marshall} described in the introduction.
The following proposition gives the stochastic interpretation for the Marshall's copulas and explains the role of the function parameters $\marphi$ and $\marpsi$. The proofs can be found in the original Marshall's paper \cite{marshall1996}.
\begin{prop}\label{prop-marshall-properties}
	Let $X, Y, Z$ be independent random variables with corresponding distribution functions $F_X, F_Y$ and $F_Z$. Define $U= \max\{ X, Z\}$ and $V=\max\{ Y, Z\}$ and let $F$ and $G$ denote their respective distribution functions. Furthermore, let $H$ be the bivariate joint distribution function of the pair $(U, V)$. Then:
	\begin{enumerate}[(i)]
		\item $F = F_XF_Z$ and $G = F_YF_Z$.
		\item A pair of functions $\marphi$ and $\marpsi$ satisfying (P1)--(P3) exists, so that $F_X(x) = \marphi(F(x))$ for all $x$, where $F(x)>0$, and $F_Y(y) = \marpsi(G(y))$ for all $y$, where $G(y)>0$.
		\item $F_Z = \dfrac{F}{F_X} = \dfrac{F}{\marphi(F)} = \dfrac{G}{F_Y} = \dfrac{G}{\marpsi(G)}$, where the expressions are defined.
		\item $H(x, y) = \marcop_{{\marphi}, {\marpsi}}(F(x), G(y))$.
		\item $\marphi^*\circ F = \marpsi^*\circ G$.
	\end{enumerate}
\end{prop}
\textbf{Observations:}
\begin{enumerate}
  \item We first observe that Conditions (P1) and (P2) mean that functions $\marphi$ and $\marpsi$ are distribution functions (if they are cadlag). However, it turns that together with (P3) they are actually continuous which is more than cadlag (see e.g.\cite{omladic2016}). Condition (P3) yields the fact that they have a reverse hazard rate which is smaller than that of a uniform random variable on $[0,1]$.
  \item According to Proposition~\ref{prop-marshall-properties}(iii),  function $\marphi$ is a distribution function which composed with $U$ yields $X$ and similarly for $\marpsi$, $V$ and $Y$.
  \item Some more stochastic interpretations of the two functions are given in Proposition~\ref{prop-marshall-properties}(iv) and (v).
  \item We shall not go into all the details. However, let us point out that Marshall in his paper \cite{marshall1996} gives a number of examples warning against overuse of the model to which one is inclined to in view of the supposedly omnipotent Sklar's theorem. In particular, only marginals satisfying the conditions of this proposition are allowed into Marshall's copula if we want to maintain the stochastic interpretation we started with.
\end{enumerate}
These facts lead our way towards generalizing Marshall's copulas in the imprecise setting. In doing so, we need to keep the stochastic interpretation in terms of shock models in power, yet we need to do it in an imprecise way. Here is a historical example due to Marshall and Olkin \cite{MaOl} which has been definitely  applied the most in the area; some of the recent applications have been cited in the introduction. An example of  stochastic interpretation will be given in Subsection \ref{subsec:stochastic} in order to help an interested reader understand the interplay of shocks in different shock induced models.
\begin{ex}\label{ex-marcop}
	If the occurrence of shocks in the model is governed by independent Poisson processes (a situation that often happens in practice), then, as it turns out, we get $X, Y$ and $Z$ to be independent with the following distribution functions:
	\begin{align*}
		F_X(x) & = 1-e^{-\lambda x}, \text{ for } x\geqslant 0 \text{ and 0 for } x<0; \\
		F_Y(y) & = 1-e^{-\eta y}, \text{ for } y\geqslant 0 \text{ and 0 for } y<0; \\
		F_Z(x) & = \begin{cases}
			0, & \text{ if } x< \mu; \\
			1, & \text{ if } x\geqslant \mu, \\
		\end{cases}
	\end{align*}
	where $\lambda, \eta$ and $\mu$ are some positive constants, actually they are the parameters of the underlying Poisson processes. Further, let $U = \max\{ X, Z \}$ and $V= \max\{ Y, Z \}$. Their distribution functions are then equal to
	\begin{align*}
		F(x) & = \begin{cases}
			1-e^{-\lambda x} &  \text{ if } x \geqslant \mu ; \\
			0 & \text{ elsewhere};
		\end{cases} \\
		G(y) & = \begin{cases}
		1-e^{-\eta y} &  \text{ if } y \geqslant \mu ; \\
		0 & \text{ elsewhere}.
		\end{cases}		
	\end{align*}
	Marshall's copula $\marcop_{{\marphi}, {\marpsi}}$ modeling the dependence between $U$ and $V$ is then generated by the functions
	\begin{align}
		\phi(u) & = \begin{cases}\label{eq-ex-phi}
			0 &  \text{ if } u = 0; \\
			1-e^{-\lambda \mu} &  \text{ if } 0 < u < 1-e^{-\lambda \mu} ; \\
			u & 1-e^{-\lambda \mu} \leqslant u \leqslant 1.
		\end{cases}
		\intertext{ and }
		\psi(v) & = \begin{cases}
		0 &  \text{ if } v = 0; \\ \nonumber
		1-e^{-\eta \mu} &  \text{ if } 0 \leqslant v < 1-e^{-\eta \mu} ; \\
		v & 1-e^{-\eta \mu} \leqslant v \leqslant 1.
		\end{cases}
	\end{align}
	Note that the above functions are only unique on $\mathrm{im}F\cup \{0\}$ and $\mathrm{im}G\cup \{0\}$ respectively.
\end{ex}

\subsection{Maxmin copulas}
Another family of copulas related to Marshall's copulas are the so called maxmin copulas introduced recently by Omladi\v{c} and Ru\v{z}i\'{c} \cite{omladic2016}. They have been shown to model the dependence between the pair of random variables \eqref{intro:eq:maxmin}, whose stochastic interpretation is described in the introduction. A maxmin copula depends on two maps $\mmphi$ and $\mmpsi\colon [0, 1]\to [0, 1]$, satisfying the properties:
\begin{enumerate}[(F1)]
	\item $\mmphi(0) = \mmpsi(0) = 0$ and $\mmphi(1) = \mmpsi(1) = 1$;
	\item $\mmphi$ and $\mmpsi$ are non-decreasing;
	\item $\mmphi^*(u) = \dfrac{\mmphi(u)}{u}\colon (0, 1]\to [1, \infty]$ and $\mmpsi_*(w) = \dfrac{1-\mmpsi(w)}{w-\mmpsi(w)}\colon [0, 1]\to [1, \infty]$ are non-increasing.
\end{enumerate}
A \emph{maxmin copula} is a map $\mmcop\colon [0, 1]\times [0, 1] \to [0, 1]$ defined by
\begin{equation*}\label{eq-maxmin-copula-defintion}
 \mmcop_{\mmphi, \mmpsi}(u, w) = uw + \min \{ u(1-w), (\mmphi(u)-u)(w-\mmpsi(w)) \}.
\end{equation*}
The following proposition gives the stochastic representation of the maxmin copulas and explains the role of functions $\mmphi$ and $\mmpsi$. The proofs can be found in Omladi\v{c} and Ru\v{z}i\'{c} \cite{omladic2016}.
\begin{prop}\label{prop-maxmin-properties}
	Let independent random variables $X, Y$ and $Z$ be given with respective distribution functions $F_X, F_Y$ and $F_Z$. Define\footnote{Note that $U$ is the same as in Proposition~\ref{prop-marshall-properties}.} $U=\max\{ X, Z\}$ and $W = \min\{ Y, Z \}$ and let $F, K$ 
	denote the distribution functions of $U$ and $W$ respectively. Let $H$ be the joint distribution function of $(U, W)$.
	Then:
	\begin{enumerate}[(i)]
		\item $F(x) = F_X(x) F_Z(x)$ and $K(y) = F_Y(y) + F_Z(y) - F_Y(y)F_Z(y)$. 
		\item A pair of functions $\mmphi$ and $\mmpsi$ satisfying (F1)--(F3) exists, so that $F_X(x) = \mmphi(F(x))$ for all $x$, where $F(x)>0$ and $F_Y(y) = \mmpsi(K(y))$ for all $y$, where $K(y)<1$.
		\item $H(x, y) = \mmcop_{\mmphi, \mmpsi}(F(x), K(y))$.
		\item $\mmphi^*\circ F = \mmpsi_*\circ K$.
		\item In terms of survival functions instead of distribution functions, the second equation in (i) assumes the following equivalent form $\hat K(y) = \hat F_Y(y) \hat F_Z(y)$. 		
	\end{enumerate}
\end{prop}
When comparing the Marshall's and maxmin models, we observe that the function $\marphi$ is defined in an analogous way corresponding to the underlying variables $X, Y$, and $Z$ while $\marpsi$ and $\mmpsi$ are defined differently, actually they are defined each in an opposite way to the other. Observe as above that functions $\marpsi$ and $\mmpsi$ are necessarily continuous.

Let us first give an adjustment of the classical Marshall-Olkin example to the maxmin case. Some more stochastic interpretation for both types of copulas will be given in Subsection \ref{subsec:stochastic}.
\begin{ex}\label{ex-mmcop}
	Let $X, Y$ and $Z$ be as in Example~\ref{ex-marcop} and take $W= \min\{Y, Z\}$. The distribution function of $W$ is then
	\[
		K(y) = \begin{cases}
			0 & \text{ if } y < 0 ; \\
		 	1-e^{-\eta y} & \text{ if } 0\leqslant y< \mu ; \\
		 	1 & \text{ if } \mu \leqslant y .
		\end{cases}
	\]
	The maxmin copula $\mmcop_{\mmphi, \mmpsi}$ modelling the dependence between $U$ and $W$ is then generated by $\phi$ as in \eqref{eq-ex-phi} and
	\[
		\mmpsi(w) = \begin{cases}
			w & \text{ if } 0 \leqslant w < 1-e^{-\eta \mu}; \\
			1-e^{-\eta \mu} & \text{ if } 1-e^{-\eta \mu} \leqslant w < 1 ; \\
			1 & \text{ if } w = 1,
		\end{cases}
	\]
	which is again unique only on $\mathrm{im} K\cup \{ 0 \}$.
\end{ex}

\subsection{Stochastic interpretation of shock model copulas}\label{subsec:stochastic}

We present in Figure 1\footnote{These images were created with the help of Mathematica \cite{math}.} stochastic interpretation of a possible concrete shock model example. On the first image possible distribution functions of three independent shocks are given: $F_X,F_Y,$ and $F_Z$. The second image (the left hand one in the second row) gives the distribution functions of the resulting component functions in the Marshall's case. The third image (the right hand one of the second row) gives the distribution functions of the resulting component functions in the maxmin case. The fact that the global shock acts beneficiary on the second component $W$ in the maxmin case results in a clear stochastic improvement in its behavior (i.e., the graph of $F_W$ is way above the graph of $F_U$) compared to the component $V$ in the Marshall's case (i.e., the graph of $F_V$ is almost always below the graph of $F_U$) -- here $U$ is the same on both second row images and the graph of $F_U$ may be seen as a prototype for stochastic behaviour of a component.

\begin{figure}[h!]\label{fig:slika1}
            \hfil \includegraphics[width=0.50\textwidth]{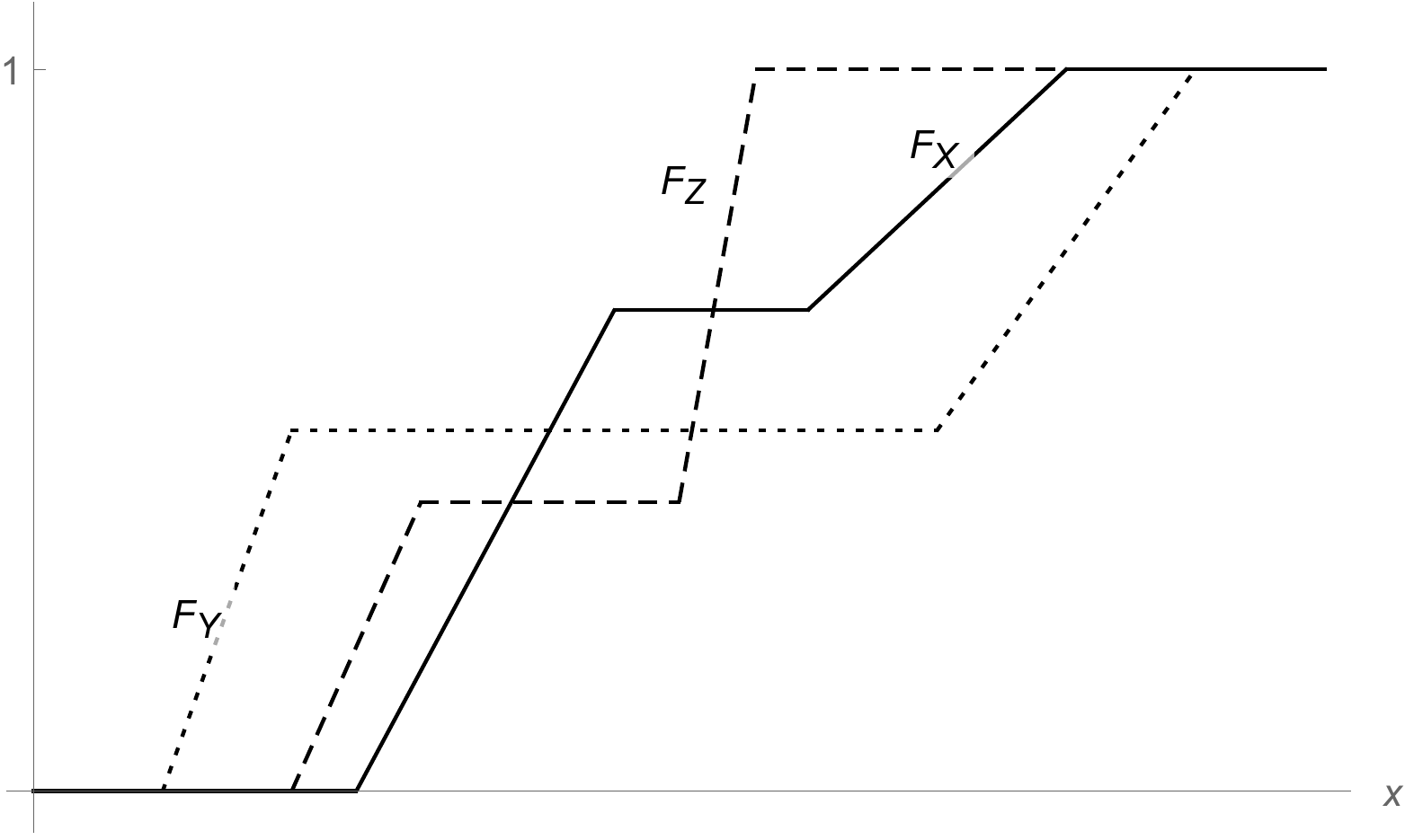} \hfil  \\ \hfil \includegraphics[width=0.49\textwidth]{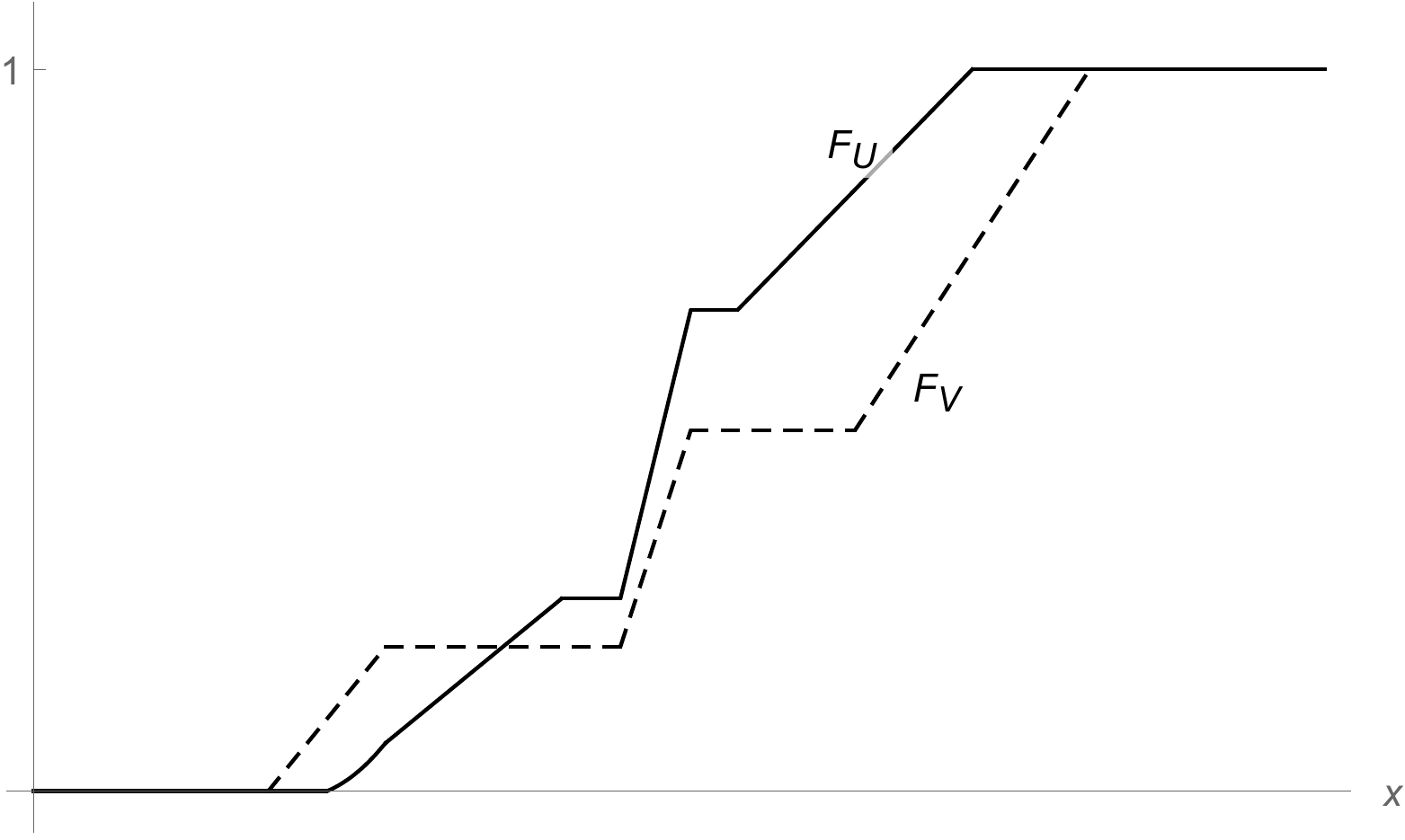} \hfil \hfil \includegraphics[width=0.49\textwidth]{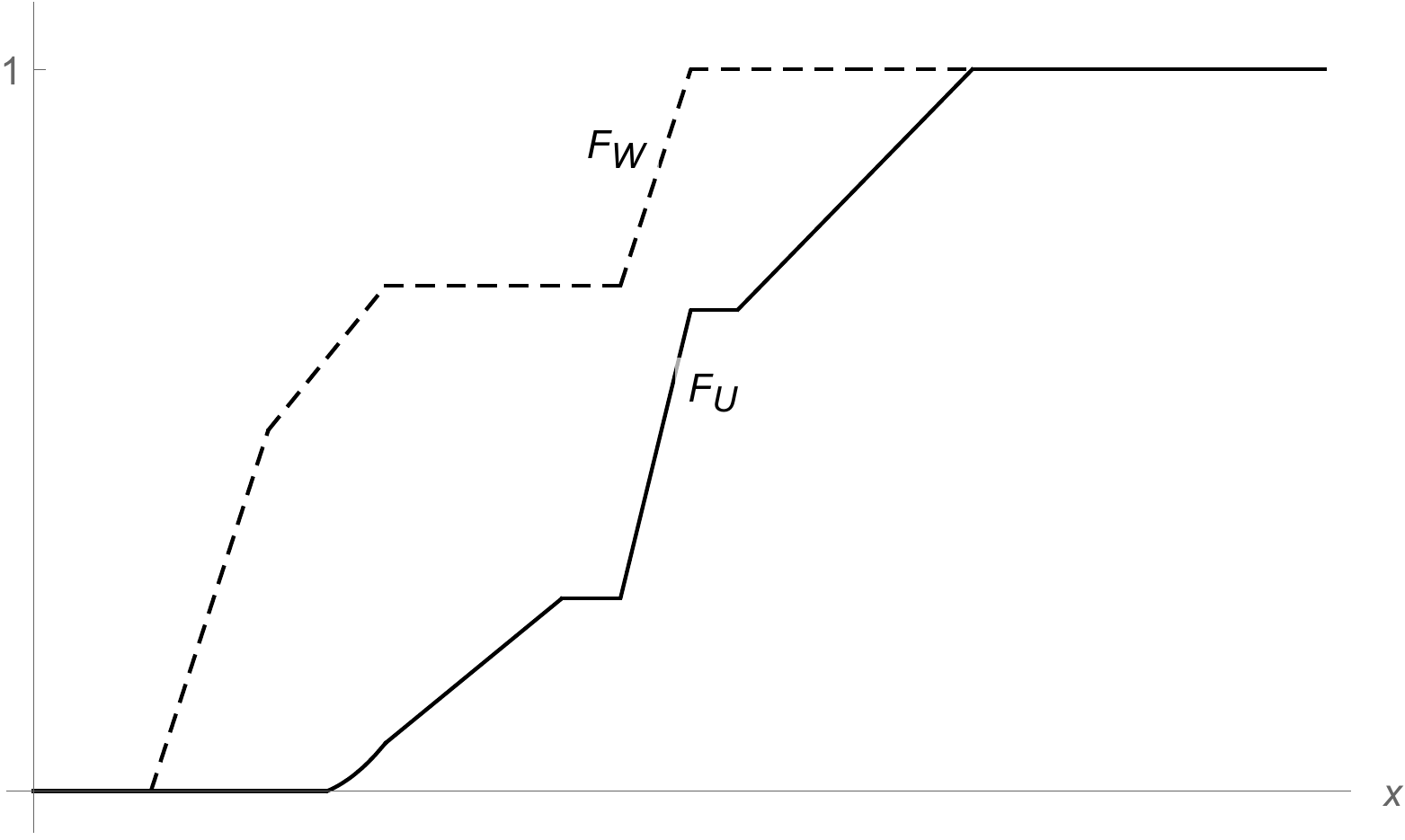}\hfil
            \caption{ Stochastic interpretation of a possible shock model realization: the Marshall's and the maxmin case   }
\end{figure}

The pair of functions $\marphi,\marpsi$ are sometimes called the \emph{generators} or \emph{generating functions} of the Marshall's copula. Similarly, the pair of functions $\marphi,\mmpsi$ are called the \emph{generators} or \emph{generating functions} of the maxmin copula. Although they are thought of as functions generating the two respective copulas independently of the marginals that one can insert into either of the copulas, let us emphasize again that their stochastic interpretation can only be given in the situation described here. So, in the case of Marshall's shock model only in this situation the generating functions $\marphi$ and $\marpsi$ can be stochastically interpreted as distribution functions of respective stochastic components of the system conditionally given that the two shocks are independent and distributed uniformly on $[0,1]$. Similarly in the case of maxmin shock model the generating functions $\mmphi$ and $\mmpsi$ have the stochastic interpretation of distribution functions of respective stochastic components of the system conditionally given that the two shocks are independent and distributed uniformly on $[0,1]$ only in the situation described.

Consequently, it is imminent that the imprecise setting goes deeper than just copulas, it has to be determined out of the mutual behavior of shocks.

\section{Order relations generated by shock models}\label{sec:order}

This section is devoted to extending the theory presented in Section \ref{ss-mmc} to the imprecise setting, i.e.\ without assuming that the underlying distribution functions are cadlag.

Let us recall briefly some details on the pairs of generating functions $\marphi$, $\marpsi$, and $\mmphi$, $\mmpsi$ (cf.\  \cite{omladic2016}). The first two are defined as functions satisfying (P1)--(P3) such that $F_X = \marphi(F)$ and $F_Y=\marpsi(G)$, where $F=F_XF_Z$ and $G=F_YF_Z$. This follows from the fact that $F$ is the distribution function of $U=\max\{X,Z\}$ and $G$ is the distribution function of $V=\max\{Y,Z\}$. Similarly $\marphi$, $\mmpsi$ are defined as functions satisfying (F1)--(F3) such that $F_X = \mmphi(F)$  and $F_Y = \mmpsi(K)$, where $K = F_Y + F_Z - F_YF_Z$. In the background this time are the distribution functions $F$ of $U=\max\{X,Z\}$ and $K$ of $V=\min\{Y,Z\}$.  It has been shown in \cite{omladic2016} that these relations uniquely determine $\marphi, \marpsi$ and $\mmpsi$ on $\mathrm{im}\, F, \mathrm{im}\, G$, and $\mathrm{im}\, K$ respectively.

Our ultimate goal is to consider the case where the variables $X$ and $Y$ have imprecise distributions, i.e.\ given in terms of $p$-boxes.  We therefore need to analyse how the order relations $\leqslant$ implied by $p$-boxes $(\low F_X, \up F_X)$ and $(\low F_Y, \up F_Y)$ translate to the order relations on the corresponding bivariate distributions, leading to the bivariate $p$-boxes and related copulas. So, the question is how the order on distribution functions $F'_X\leqslant F_X$ and $F'_Y\leqslant F_Y$, transmits to order relation $\leqslant$ on the corresponding pairs of generating functions $\marphi, \marpsi$ in the case of Marshall's copulas, respectively $\mmphi,\mmpsi$ in the case of maxmin copulas. As the construction for $\marpsi$ is essentially the same as the one for $\marphi$, we will analyse only the cases of $\marphi$ and $\mmpsi$.


\subsection{Order relations for Marshall's copulas}\label{s-ormc}
%
Here we denote by $f(x+)$, respectively $f(x-)$, the right limit, respectively the left limit of a monotone (increasing) function $f$ at $x$; observe that existence of these limits follows by monotonicity of the function. For distribution functions $F_X$ and $F_Z$ let $F=F_XF_Z$. Choose a $u\in (0, 1)$ and let $x_0$ be any value such that $F(x_0-) \leqslant u \leqslant F(x_0+)$.
Furthermore, let us introduce
\begin{align*}
	u_- & = F_X(x_0-)F_Z(x_0-) = F(x_0-) & u_l & = F_X(x_0-)F_Z(x_0) \\
	u_+ & = F_X(x_0+)F_Z(x_0+) = F(x_0+) & u_u & = F_X(x_0+)F_Z(x_0)
\end{align*}
and define
\begin{equation}\label{eq-phi-ext}
	\mmphi(u) =
	\begin{cases}
		0 & \text{if } u = 0; \\
		F_X(x_0-) & \text{if } u_-\leqslant u \leqslant u_l; \\
		\dfrac{u}{F_Z(x_0)} & \text{if } u_l\leqslant u \leqslant u_u; \\
		F_X(x_0+) & \text{if } u_u\leqslant u \leqslant u_+ ; \\
		1 & \text{if } u = 1.
	\end{cases}
\end{equation}
We will show in Proposition \ref{prop:marshall}\emph{(iii)} that this definition fulfills the requirement that $F_X = \marphi(F)$. However, this requirement does not determine function $\marphi$ uniquely in general, only the values on the image of $U$ are determined. A possible definition, based on the extension from the image to the entire interval $[0, 1]$ via the linear interpolation technique, has been proposed in \cite{omladic2016}. This extension, however is not suitable for our purpose, because \textbf{(a)} it assumes distribution functions to be cadlag, while our distribution functions are monotone only in accordance with the usual assumption in the $p$-box approach, and \textbf{(b)} the order on distribution functions such as $F_X$ is not preserved to the corresponding generating functions $\marphi$.
\begin{prop}\label{prop:marshall}
	Let $F_X, F_Z$ be given, and let $F=F_XF_Z$. If $\mmphi$ is defined by \eqref{eq-phi-ext}, then
	\begin{enumerate}[(i)]
		\item $\mmphi$ is well defined;
		\item $\mmphi$ is continuous;
		\item $\mmphi(F(x)) = F_X(x)$ for every $x\in \RR$ such that $F(x)>0$;
		\item $\mmphi$ is non-decreasing;
		\item $\mmphi^*(u) = \dfrac{\mmphi(u)}{u}$ is non-increasing.
	\end{enumerate}
\end{prop}
\begin{proof}
	To see (i) choose $u\in (0, 1)$ and assume that for some real values $x_0 < x_1$ we have $F(x_0-)\leqslant u \leqslant F(x_0+)$ and $F(x_1-)\leqslant u \leqslant F(x_1+)$. This is only possible if $F(x_0+) = u = F(x_1-)$, which also means that $F_X(x_0+)=F_X(x_1-) = \mmphi(u)$. Note also that $\dfrac{u_l}{F_Z(x_0)} = F_X(x_0-)$ and $\dfrac{u_u}{F_Z(x_0)} = F_X(x_0+)$, which means that $\mmphi$ is also well defined within every interval $[u_-, u_+]$.
	
	To prove (ii), first observe that $\mmphi$ is continuous within every interval $[u_-, u_+]$.
	It is also clear that $\mmphi(F(x_0-)) = \mmphi(F(x_0-)-) = F_X(x_0-)$ and $\mmphi(F(x_0+)) = \mmphi(F(x_0+)+) = F_X(x_0+)$, which makes $\mmphi$ continuous everywhere on $(0, 1)$.
		
	Now choose an $x$ such that $F(x)=F_X(x) F_Z(x)>0$. Then clearly, $F_Z(x)>0$.
	Moreover, $F(x)\in [F_X(x-)F_Z(x), F_X(x+)F_Z(x)] = [u_l, u_u]$, whence $\mmphi(F(x)) = \dfrac{F(x)}{F_Z(x)} = F_X(x)$, which confirms (iii).
	
	Next choose an $u<u'$. Then for the corresponding $x_0$ and $x'_0$ we clearly have that $x_0 \leqslant x'_0$. In the case where $x_0 = x'_0$, $\mmphi(u) \leqslant \mmphi(u')$ follows directly from \eqref{eq-phi-ext}. If $x_0 < x'_0$, then we have that $\mmphi(u) \leqslant F_X(x_0+) \leqslant F_X(x'_0-) \leqslant \mmphi(u')$, which proves (iv).
	
	To show (v) we choose $u<u'$ as above. Now clearly, $\mmphi^*$ is non-increasing on the intervals $[u_-, u_+]$ and $[u'_-, u'_+]$, whence $\mmphi^*(u)\geqslant \mmphi^*(u')$ if $x_0 = x'_0$. If $x_0 < x'_0$, then we have
	\begin{multline*}
		\mmphi^*(u) \geqslant \mmphi^*(u_+) = \dfrac{F_X(x_0+)}{F_X(x_0+)F_Z(x_0+)} = \dfrac{1}{F_Z(x_0+)} \\
		\geqslant \dfrac{1}{F_Z(x'_0-)} = \mmphi^*(u'_-) \geqslant \mmphi^*(u').
	\end{multline*}	
\end{proof}

\begin{lem}\label{lem-phi-order}
	Let $F'_X\leqslant F_X$ and $F_Z$ be given, and let $F=F_XF_Z$ and $F'=F'_XF_Z$. Then $\mmphi' \leqslant \mmphi$, where $\mmphi'$ and $\mmphi$ are defined by applying \eqref{eq-phi-ext} to $F'_X$ and $F_X$ respectively.
\end{lem}
\begin{proof}
	Let $u\in (0, 1)$, $x_0$, and $x_1$ be such that $F(x_0-)\leqslant u \leqslant F(x_0+)$ and $F'(x_1-)\leqslant u \leqslant F'(x_1+)$. Since $F'\leqslant F$, we may assume that $x_0 \leqslant x_1$. We consider two cases.
	
	\emph{Case 1}: $x_0 = x_1$. In this case we have that $F'_X(x_0-) \leqslant F_X(x_0-)$ and $F'_X(x_0+) \leqslant F_X(x_0+)$. Define
	\begin{equation*}
		\mmphi_0(u) = \min \left\{ F_X(x_0+), \max\left\{ \frac{u}{F_Z(x_0)}, F_X(x_0-) \right\} \right\}
	\end{equation*}
	and similarly for $\mmphi'_0$. Then clearly $\mmphi(u) = \mmphi_0(u) \leqslant \mmphi'_0(u) = \mmphi'(u)$ and the desired conclusion follows.
	
	\emph{Case 2}: $x_0 < x_1$. Define a distribution function $F^u$ by
	\begin{equation*}
		F^u(x) =
		\begin{cases}
			\mmphi(u)F_Z(x) & \text{if } x_0\leqslant x \leqslant x_1 \\
			F(x) & \text{otherwise.}
		\end{cases}
	\end{equation*}
	To see that $F^u$ is indeed a distribution function note that $F^u$ is clearly non-decreasing on $(x_0, x_1)$ and
	\begin{align*}
		F^u(x_0+) & \geqslant \mmphi(u)F_Z(x_0) \geqslant F_X(x_0-)F_Z(x_0) \geqslant F(x_0-) = F^u(x_0-) \\
		\intertext{and}
		F^u(x_1-)& \leqslant \mmphi(u)F_Z(x_1) \leqslant F_X(x_0+)F_Z(x_1) \\
		 & \leqslant F_X(x_1+)F_Z(x_1) \leqslant F(x_1+) = F^u(x_1+).
	\end{align*}	
	Everywhere else $F^u$ coincides with $F$, which is a distribution function.
	
	We now show that
\begin{equation}\label{eq:phiulephi}
  \mmphi^u(u) \leqslant \mmphi(u),
\end{equation}
where $\mmphi^u$ is obtained from $F^u$ in the same way as $\mmphi$ was obtained from $F$. We first consider the case where $u_-\leqslant u\leqslant u_u$. Then $\mmphi(u)\geqslant \dfrac{u}{F_Z(x_0)}$, whence $F^u(x_0+) \geqslant \mmphi(u)F_Z(x_0)\geqslant u$, and
	\begin{equation*}
		F^u(x_0+) = \mmphi(u)F_Z(x_0+) \leqslant F_X(x_0+)F_Z(x_0+) =  F(x_0+).
	\end{equation*}
	Since, clearly, $F^u(x_0-) = F(x_0-)$, we can use the argument of Case 1 to get \eqref{eq:phiulephi}.
	
	Secondly, if $u_u\leqslant u \leqslant u_+$, then we have that $\mmphi(u) = F_X(x_0+)$ and define
	\begin{equation*}
		u' = F^u(x_1) = \mmphi(u)F_Z(x_1) = F_X(x_0+)F_Z(x_1) \geqslant F_X(x_0+)F_Z(x_0+) = u_+ \geqslant u.
	\end{equation*}
	 By monotonicity of $\mmphi$ and $\mmphi^u$, and the fact that $u'\in \mathrm{im}F^u$, we then have that
	 \begin{equation*}
	 	\mmphi^u(u)\leqslant \mmphi^u(u') = \dfrac{u'}{F_Z(x_1)} = \dfrac{F_X(x_0+)F_Z(x_1)}{F_Z(x_1)} = F_X(x_0+) = \mmphi(u).
	 \end{equation*}	
so that \eqref{eq:phiulephi} holds again.

We now also prove that
\begin{equation}\label{eq:phi'lephiu}
  \mmphi'(u)\leqslant \mmphi^u(u).
\end{equation}
To see this, first notice that $F'(x_1+) \leqslant F(x_1+) = F^u(x_1+)$. Next we want to show that also
\[
    F'(x_1-) \leqslant F^u(x_1-)
\]
By the choice of $x_1$ we observe that $F'(x_1-) \leqslant u$. To see that $F^u(x_1-) \geqslant u
$, first consider the case when $u\leqslant u_u$, so that $\mmphi(u) \geqslant \dfrac{u}{F_Z(x_0)}$, and consequently
\[
F^u(x_1-) = \mmphi(u) F_Z(x_1-) \geqslant \dfrac{u}{F_Z(x_0)}F_Z(x_1-) \geqslant u,
\]
since clearly, $F_Z(x_1-) \geqslant F_Z(x_0)$. Now, if $u\geqslant u_u$, then $\mmphi(u) = F(x_0+)$, and therefore
\[
F^u(x_1-) = F_X(x_0+) F_Z(x_1-) \geqslant  F_X(x_0+)F_Z(x_0+) = u_+ \geqslant u.
\]
	
	 We have thus shown that both $F^u(x_1+) \geqslant F'(x_1+)$ and $F^u(x_1-) \geqslant F'(x_1-)$, whence by applying the argument of Case 1 of this proof, we derive Inequality \eqref{eq:phi'lephiu}.
	 Together with \eqref{eq:phiulephi} this implies the desired conclusion $\mmphi(u) \geqslant \mmphi'(u)$ and completes the proof.
\end{proof}

\textbf{Observations}:
\begin{enumerate}
  \item Due to the symmetry between $\marphi$ and $\marpsi$ everything that was done in this subsection for distribution functions $F_X$ and $F_Z$ in relation to them giving rise to the generating function $\marphi$, holds also for distribution functions $F_Y$ and $F_Z$ in relation to them giving rise to the generating function $\marpsi$.
  \item Using the techniques prepared in this Subsection we can see that Proposition \ref{prop-marshall-properties} remains valid in the imprecise case. Indeed, Item \emph{(i)} is simply given by definition, Item \emph{(ii)} follows by Proposition \ref{prop:marshall}\emph{(iii)}, and the rest of it follows easily form these two facts, using also the above Observation 1.
  \item We have thus seen that Marshall's copulas are technically exactly the same objects in the imprecise setting as in the classical approach, we only need to be careful in choosing the right generators. Additional care has to be taken about the order and for that we need Lemma \ref{lem-phi-order}.
\end{enumerate}

\subsection{Order relations for maxmin copulas}\label{s-ormm}

In this subsection we consider distribution functions $F_Y, F_Z$ and $K = F_Y + F_Z - F_YF_Z$. 
Observe that the way the ordered distribution functions $F_X$ and $F_Z$ give rise to ordered functions $\mmphi$ is exactly the same as in Subsection \ref{s-ormc}. It remains to determine how ordered distribution functions $F_Y$ and $K$ determine the corresponding ordered generating functions $\mmpsi$ of maxmin copulas using the defining relation $\mmpsi(K) = F_Y$ of $\mmpsi$.


To avoid repeating the tedious procedures from the previous section for this case, we use the transformation that translates our case to the one analysed there.
Let for every distribution function $F$ define its reverse distribution function $\tilde F(x) = 1-F(-x)$.
It is easy to verify that
\[
	\tilde K = \tilde F_Y \tilde F_Z.
\]
We are introducing notation $\tilde F$ and the term reverse distribution function only for the sake of simplifying the procedure of this subsection. (Observe in passing that $F\mapsto {\tilde F}$ sends a cadlag function to a caglad function, while a monotone nondecreasing function, like the distribution functions we are dealing with, is sent simply to a monotone nondecreasing function.) Now it only remains to translate the expressions used in the previous section.
First take the equation $u = \tilde K(x)$, which becomes $w = K(y)$ by replacing $w = 1-u$ and $y = -x$.
Thus, given a $w\in (0, 1)$ we let $y_0$ to be any value such that $K(y_0-) \leqslant w \leqslant K(y_0+)$.
Definition \eqref{eq-phi-ext} now directly translates into:
\begin{equation}\label{eq-chi-ext}
\mmpsi(w) =
\begin{cases}
0 & \text{if } w = 0; \\
F_Y(y_0-) & \text{if } w_-\leqslant w \leqslant w_l; \\
\dfrac{w-F_Z(y_0)}{1-F_Z(y_0)} & \text{if } w_l\leqslant w \leqslant w_u; \\
F_Y(y_0+) & \text{if } w_u\leqslant w \leqslant w_+ ; \\
1 & \text{if } w = 1,
\end{cases}
\end{equation}
where
\begin{align*}
w_- & = F_Y(y_0-) + F_Z(y_0-) -F_Y(y_0-)F_Z(y_0-) = K(y_0-) \\ w_l & = F_Y(y_0-) + F_Z(y_0) -F_Y(y_0-)F_Z(y_0) \\
w_+ & = F_Y(y_0+) + F_Z(y_0+) -F_Y(y_0+)F_Z(y_0+) = K(y_0+) \\ w_u & = F_Y(y_0-) + F_Z(y_0) -F_Y(y_0-)F_Z(y_0).
\end{align*}
In fact, we could write directly, that
\begin{equation}\label{eq-tilde-phi-chi}
	\mmpsi(w) = 1-\tilde \marphi(1-w),
\end{equation}
where $\tilde \marphi$ satisfying $\tilde \marphi(\tilde K) = \tilde F_Y$ is obtained as in the previous section.
Moreover, we note that the following relation also holds:
\begin{equation}
	\frac{1}{\mmpsi_*(w)} = 1-\frac{1}{\tilde \marphi^*(1-w)},
\end{equation}
which among others shows that if $\tilde \marphi^*$ is nonincreasing, so is $\mmpsi_*$.

Thus, we can state the following corollary.
\begin{cor}\label{cor}
	Let $F_Y, F_Z$ and $K=F_Y + F_Z - F_YF_Z$ be given and let $\mmpsi$ be defined by \eqref{eq-chi-ext}. Then
	\begin{enumerate}[(i)]
		\item $\mmpsi$ is well defined;
		\item $\mmpsi$ is continuous;
		\item $\mmpsi(K(y)) = F_Y(y)$ for every $y\in \RR$ such that $K(y)<1$;
		\item $\mmpsi$ is non-decreasing;
		\item $\mmpsi_*(w) = \dfrac{1-\mmpsi(w)}{w-\mmpsi(w)}$ is non-increasing.
	\end{enumerate}
\end{cor}

\begin{lem}\label{lem-chi-order}
	Let $F'_Y\leqslant F_Y$ and $F_Z$ be given, and let $K=F_Y + F_Z - F_YF_Z$ and $K'=F'_Y + F'_Z - F'_YF'_Z$. Then $\mmpsi' \leqslant \mmpsi$, where $\mmpsi'$ and $\mmpsi$ are defined by applying \eqref{eq-chi-ext} to $F'_Y$ and $F_Y$ respectively.
\end{lem}
\begin{proof}
	Note that $F'_Y \leqslant F_Y$ implies $\tilde F'_Y \geqslant \tilde F_Y$, and therefore, by Lemma~\ref{lem-phi-order}, $\tilde \marphi'\geqslant \tilde \marphi$, and by the relation \eqref{eq-tilde-phi-chi} this implies that $\mmpsi'\leqslant \mmpsi$.
\end{proof}

\textbf{Observations}:
\begin{enumerate}
  \item The generating function $\mmphi$ is treated simply via the methods of Subsection \ref{s-ormc}.
  \item We see that using Corollary \ref{cor} and Observation 1 we conclude the validity of Proposition \ref{prop-maxmin-properties} in the imprecise setting.
  \item We conclude also in this case that maxmin copulas as objects are the same in the imprecise setting as they were in the classical case, we only need to be careful about the definition of the generators, while their order is being taken care of by Lemma \ref{lem-chi-order}.
\end{enumerate}

\subsection{Associated generating functions }
Let $F_X,F_Y$ and $F_Z$, and also $F, G$ and $K$ be as in Subsections~\ref{s-ormc} and \ref{s-ormm}, and let $\marphi, \marpsi$, and $\mmpsi\colon [0, 1]\to [0, 1]$ be such that $\marphi(F) = F_X, \marpsi(G) = F_Y$ and $\mmpsi(K) = F_Y$. If they also satisfy the corresponding Conditions (P1)--(P3) and (F1)--(F3), then we will say that the triple $(\marphi,\marpsi,\mmpsi)$ is \emph{associated} to the corresponding $F, G$, and $K$ given $F_Z$, or simply that it is \emph{associated} to the triple $(F_X, F_Y, F_Z)$. We will also say that a single function $\marphi, \marpsi$ or $\mmpsi$ is associated to a triple $(F_X, F_Y, F_Z)$ if it satisfies the above conditions. Note that in general there may be multiple generating functions associated to some triple of distribution functions. Namely, the above requirements only determine their values on the images of the corresponding distribution functions.

Let a triple of distribution functions $(F_X, F_Y, F_Z)$ be given and let $\mathbf \Phi$, respectively $\mathbf{\Psi}$, respectively $\mathbf{X}$, be sets of generating functions $\marphi$, respectively $\marpsi$, respectively $\mmpsi$, so that all triples $(\marphi, \marpsi, \mmpsi)$ are associated to $(F_X, F_Y, F_Z)$.
 Denote $\marphi_{\mathrm{min}} = \inf_{\marphi\in\Phi} \marphi$ and $\marphi_{\mathrm{max}} = \sup_{\marphi\in\Phi} \marphi$ and adjoin these two functions to $\mathbf{\Phi}$ without changing its notation (this might be an abuse of notation in case that the two functions had not belonged to the set to start with). Similarly, we adjoin  $\marpsi_{\mathrm{min}}$ and $\marpsi_{\mathrm{max}}$ the respective infimum and supremum of the set $\mathbf \Psi$  to it, and $\chi_{\mathrm{min}}$ and $\chi_{\mathrm{max}}$ the respective infimum and supremum of the set $\mathbf X$ to it.
\begin{prop}\label{prop-inf-sup-phi}
	Every triple $\{(\marphi,\marpsi,\mmpsi);\marphi\in\mathbf{\Phi}, \marpsi\in\mathbf{\Psi}, \mmpsi\in\mathbf{X}\}$ is associated to the triple $(F_X,F_Y,F_Z)$.
\end{prop}
\begin{proof}
	It only remains to show that $(\marphi_{\mathrm{min}}, \marpsi_{\mathrm{min}}, \mmpsi_{\mathrm{min}})$ and   $(\marphi_{\mathrm{max}}, \marpsi_{\mathrm{max}}, \mmpsi_{\mathrm{max}})$ are associated to $(F_X,F_Y,F_Z)$ if all other elements of the sets $\mathbf \Phi, \mathbf \Psi$ and $\mathbf X$ are.
	
	As an infimum or supremum of any family of increasing functions is increasing, $\marphi_{\mathrm{min}}$ and $\marphi_{\mathrm{max}}$ are increasing, and so are $\marphi^*_{\mathrm{min}}(u) =  \dfrac{\marphi_{\mathrm{min}}(u)}{u}$ and $\marphi^*_{\mathrm{max}}(u) = \dfrac{\marphi_{\mathrm{max}}(u)}{u}$, since $\phi^*_{\mathrm{min}} = \inf_{\phi\in \mathbf \Phi}\phi^*$ and ${\phi}^*_{\mathrm{max}} = \inf_{\phi\in \mathbf \Phi}\phi^*$. Similar argument can be used for ${\chi_*}_{\mathrm{min}}(w) = \dfrac{1-\mmpsi_{\mathrm{min}}(w)}{w-\mmpsi_{\mathrm{min}}(w)}$ and ${\chi_*}_{\mathrm{max}}(u) = \dfrac{1-\mmpsi_{\mathrm{max}}(w)}{w-\mmpsi_{\mathrm{max}}(w)}$. 	
	
	It is also straightforward  to see that ${\marphi_{\mathrm{min}}}(F)={\marphi_{\mathrm{max}}}(F)=\marphi(F) =F_X$, for every $\marphi \in \mathbf \Phi$ and similarly for $\mmpsi_{\mathrm{min}}$ and ${\mmpsi}_{\mathrm{max}}$.
\end{proof}

\section{Imprecise Marshall's copulas and maxmin copulas}\label{sec:imprecise}

We are now in position to extend the notion of Marshall's copulas and maxmin copulas to the imprecise probability setting. There is no unique way to do it. One might want to consider imprecise copulas of some kind and insert imprecise marginals into them. However, this approach may not lead to the desired solution, since the main point of these families of copulas is that they are induced by shock models. We want to extend these two notions so that this main property would remain true in the imprecise setting as well. More precisely, if a bivariate distribution $\marcop_{{\marphi}, {\marpsi}}(F, G)$ describes a shock model of Marshall type, then we want to have something along the line of Proposition~\ref{prop-marshall-properties}, especially Condition (iii), to hold; and similarly for the case of maxmin copulas. Applying an imprecise copula representing a set of precise copulas to marginals given in terms of $p$-boxes, would then correspond to applying a set of copulas to a set of marginals, whereas only some of the obtained models would have interpretation in terms of shock models.

These are the reasons why we decided for a different approach. Instead of constructing a general abstract imprecise copula, which would  correspond to cases of interest only in some selected cases, we allow imprecision in the underlying shock models, and then analyse how the obtained model relates to the theory of imprecise copulas. The shocks that we denote by $X, Y$ and $Z$ will now be allowed to have imprecise distribution functions given in terms of $p$-boxes. In fact, for technical reasons, we will only allow $X$ and $Y$ to have imprecise distributions, while $Z$ will still have a precise distribution function\footnote{It has been seen in Subsection \ref{s-ormc} respectively \ref{s-ormm} how difficult it is to define the generating function $\marphi$ respectively $\mmpsi$ so to maintain its order. If we let $Z$ imprecise as well and let, say, $F_Z \le F'_Z$ and $F_X \le F'_X$ then it would be hard to expect that $\marphi \le \marphi'$ in general or vice-versa, no matter what definition of this generator we choose fulfilling the other conditions. So, maintaning the definition and order of the generators and consequently  the structure of (bivariate) $p$-boxes seems to be a much greater challenge in this case.}. So, we let $(\low F_X, \up F_X)$ and $(\low F_Y, \up F_Y)$ be $p$-boxes describing the knowledge about the distribution of variables $X$ and $Y$. The precise distribution function of $Z$ is denoted by $F_Z$.

Denote $\low F = \low F_X F_Z, \up F = \up F_X F_Z, \low G = \low F_Y F_Z, \up G = \up F_X F_Z, \low K = \low F_Y + F_Z  - \low F_YF_Z$ and $\up K = \up F_Y + F_Z  - \up F_YF_Z$. Let

\begin{equation*}\label{eq-low-phi}
  \begin{split}
	\low{\phi} & = \inf\{ \phi \colon \phi \text{ is associated to $\low F$ given $F_Z$} \} \\
	\up{\phi} & = \sup\{ \phi \colon \phi \text{ is associated to $\up F$ given $F_Z$} \}.
\end{split}
\end{equation*}
By Proposition~\ref{prop-inf-sup-phi}, $\low{\phi}$ is the minimal function associated to $\low F$ and $\up \phi$ the maximal function associated to $\up F$. Similarly we define $\low{\psi}, \up{\psi}, \low{\chi}$ and $\up{\chi}$.

The following proposition is an easy consequence of Lemmas~\ref{lem-phi-order} and \ref{lem-chi-order}.

\begin{prop}\label{prop-order}
	Let $(\low F_X, \up F_X)$ and $(\low F_Y, \up F_Y)$ be $p$-boxes representing the available information on the distribution functions of random variables $X$ and $Y$, and $F_Z$ the distribution function for $Z$. Further let $F_X$ and $F_Y$ be distribution functions such that $\underline F_X \leqslant  F_X \leqslant  \overline F_X$ and $\underline F_Y \leqslant  F_Y \leqslant  \overline F_Y$. Denote also $F = F_X F_Z, G = F_YF_Z$ and $K = F_Y + F_Z  - F_YF_Z$. Then:
	\begin{enumerate}[(i)]
		\item $\low F \leqslant F \leqslant \up F, \low G \leqslant G \leqslant \up G$ and $\low K \leqslant K \leqslant \up K$;
		\item There exist $\phi, \psi$ and $\chi\colon [0, 1]\to [0, 1]$ associated to $F, G$ and $K$ respectively, such that $\low \phi \leqslant \phi \leqslant \up \phi, \low \psi \leqslant \psi \leqslant \up \psi$ and $\low \chi \leqslant \chi \leqslant \up \chi$.
	\end{enumerate}	
\end{prop}
\subsection{Imprecise Marshall's copulas}
Proposition~\ref{prop-order} suggests that the dependence of random variables $U=\max\{ X, Z\}$ and $V=\max\{ Y, Z \}$ can be modeled by a family of copulas whose order corresponds to the order of the distribution functions of $X$ and $Y$. Thus, if the distributions are modeled by $p$-boxes $(\low F_X, \up F_X)$ and $(\low F_Y, \up F_Y)$, the corresponding functions $\marphi$ and $\marpsi$ belong to intervals of the form $(\low \marphi, \up{\marphi})$ and $(\low \marpsi, \up{\marpsi})$. This justifies the following definition.
\begin{defn}\label{def-i-mar-cop}
	The family of copulas
	\begin{equation}\label{eq-i-mar-cop}
		\marsetcop = \{ \marcop_{\marphi, \marpsi} \colon \low{\marphi}\leqslant  \marphi \leqslant  \up{\marphi}, \low{\marpsi}\leqslant  \marpsi \leqslant  \up{\marpsi}\},
	\end{equation}
	where $\low{\marphi} \leqslant  \up\marphi$ and $\low{\marpsi}\leqslant  \up{\marpsi}$, and all $\marphi$ and $\marpsi$, including the bounds, satisfy conditions (P1)--(P3),
	is called an \emph{imprecise Marshall's copula}.
\end{defn}
\begin{prop}
	Let $\marsetcop$ be an imprecise Marshall's copula of the form \eqref{eq-i-mar-cop}. Then it contains the minimal and the maximal element with respect to pointwise ordering:
	\begin{align*}
	\min_{\marcop_{\marphi, \marpsi}\in\marsetcop}\marcop_{\marphi, \marpsi}  & = \marcop_{\low\marphi, \low\marpsi}; \\
	\max_{\marcop_{\marphi, \marpsi}\in\marsetcop}\marcop_{\marphi, \marpsi}  & = \marcop_{\up\marphi, \up\marpsi},
	\end{align*}
	i.e. $\marcop_{\low\marphi, \low\marpsi}(u, v) \leqslant \marcop_{\marphi, \marpsi} (u, v)  \leqslant \marcop_{\up\marphi, \up\marpsi}(u, v)$ for every copula $\marcop_{\marphi, \marpsi}\in \marsetcop$ and every $u, v\in [0, 1]$.
\end{prop}
\begin{proof}
	We only need to prove that the order on the generating functions translates to the order on the copulas. So, take some $\marphi\leqslant \marphi'$ and $\marpsi\leqslant \marpsi'$ and calculate:
	\begin{equation}\label{eq-order-marcop}
		\marcop_{\marphi, \marpsi}(u, v) = uv\min\left\{ \frac{\marphi(u)}{u}, \frac{\marpsi(v)}{v} \right\} \leqslant  uv\min\left\{ \frac{\marphi'(u)}{u}, \frac{\marpsi'(v)}{v} \right\} = \marcop_{\marphi', \marpsi'}(u, v).
	\end{equation}
	Taking the minimal and maximal functions respectively, thus clearly gives the minimal and the maximal Marshall's copula of $\marsetcop$.
\end{proof}
\begin{rem}\label{rem:marshall}
  Observe that in this case the set of copulas of Equation \eqref{eq-i-mar-cop} actually contains the lower and the upper bound so that the two bounds are necessarily copulas unlike in the general case of Equation \eqref{intermediate} where we may encounter a problem mentioned immediately following that equation. Note however, that not every copula lying between $\marcop_{\low{\marphi}, \low{\marpsi}}$ and $\marcop_{\up{\marphi}, \up{\marpsi}}$ is necessarily a Marshall's copula.
\end{rem}

\begin{cor}\label{cormar}
  The pair of copulas $(\marcop_{\low\marphi, \low\marpsi},\marcop_{\up\marphi, \up\marpsi})$ is an imprecise copula in the sense of Montes et al.\ \cite{montes2015} satisfying Condition (C).
\end{cor}
\subsection{Imprecise maxmin copulas}
Similarly as above we define an imprecise maxmin copula.
\begin{defn}\label{def-i-mm-cop}
	The family of copulas
	\begin{equation}\label{eq-i-mm-cop}
	\mmsetcop = \{ \mmcop_{\mmphi, \mmpsi} \colon \low{\mmphi}\leqslant  \mmphi \leqslant  \up{\mmphi}, \low{\mmpsi}\leqslant  \mmpsi \leqslant  \up{\mmpsi}\},
	\end{equation}
	where $\low{\mmphi} \leqslant  \up\mmphi$ and $\low{\mmpsi}\leqslant  \up{\mmpsi}$, and all $\mmphi$ and $\mmpsi$, including the bounds, satisfy conditions (F1) -- (F3),
	is called an \emph{imprecise maxmin copula}.
\end{defn}
\begin{prop}\label{prop-maxmin-bounds}
	Let $\mmsetcop$ be an imprecise maxmin copula of the form \eqref{eq-i-mm-cop}. Then it contains the minimal and the maximal elements with respect to pointwise ordering:
	\begin{align*}
	\min_{\mmcop_{\mmphi, \mmpsi}\in\mmsetcop}\mmcop_{\mmphi, \mmpsi} (u, v) & = \mmcop_{\low\mmphi, \up\mmpsi}; \\
	\max_{\mmcop_{\mmphi, \mmpsi}\in\mmsetcop}\mmcop_{\mmphi, \mmpsi} (u, v) & = \mmcop_{\up\mmphi, \low\mmpsi} ,
	\end{align*}
	i.e. $\mmcop_{\low\mmphi, \up\mmpsi} \leqslant \mmcop_{\mmphi, \mmpsi} (u, v)  \leqslant \mmcop_{\up\mmphi, \low\mmpsi}$ for every copula $\mmcop_{\mmphi, \mmpsi}\in \mmsetcop$ and every $u, v\in [0, 1]$.
\end{prop}
\begin{proof}
	Recall the definition $\mmcop_{\mmphi, \mmpsi}(u, v) = uv + \min \{ u(1-v), (\mmphi(u)-u)(v-\mmpsi(v)) \}$, and that $\mmphi(u)-u\geqslant 0$ for every $u$ and $v-\mmpsi(v)\geqslant 0$ for every $v$. It follows immediately that the minimum is attained at the pair $(\low{\mmphi}, \up{\mmpsi})$ and the maximum at the pair $(\up{\mmphi}, \low{\mmpsi})$.
\end{proof}

\begin{cor}\label{cormm}
  The pair of copulas $(\mmcop_{\low\mmphi, \up\mmpsi},\mmcop_{\up\mmphi, \low\mmpsi} )$ is an imprecise copula in the sense of Montes et al.\ \cite{montes2015} satisfying Condition (C).
\end{cor}

Now we relate the imprecise copulas with the shock models with imprecise underlying distributions, modelled by $p$-boxes.
\begin{prop}\label{prop-independent-product-distribution}
	Let $X$ and $Y$ be independent random variables whose distributions are given imprecisely in terms of $p$-boxes $(\low F_X,  \up{F}_X)$ and $(\low F_Y,  \up{F}_Y)$ respectively. Then
	\begin{enumerate}[(i)]
		\item the distribution function of the random variable $\max\{ X, Y \}$ can be given in terms of the $p$-box $(\low F_X \low F_Y, \up F_X\up F_Y)$;
		\item the distribution function of the random variable $\min\{ X, Y \}$ can be given in terms of the $p$-box $(\low F_X + \low F_Y - \low F_X \low F_Y, \up F_X + \up F_Y- \up F_X \up F_Y)$.
	\end{enumerate}
\end{prop}
\begin{proof}
	From the assumptions we obtain:
	\begin{align*}
		\low P(\max\{X, Y\}\leqslant  x) & = \low P(X\leqslant  x, Y\leqslant  x) \\
		 & = \low F(x, x) \\
		 & = \low F_X(x) \low F_Y(x)
	\end{align*}
	 and similarly for the upper bounds
	 \begin{equation*}
	 	\up P(\max\{X, Y\}\leqslant  x) = \up F_X(x)\up F_Y(x).
	 \end{equation*}
	 Together, the above equalities prove (i).
	
	 Note that $P(-X < -x) = 1-F_X(x) = \hat F(x)$ and therefore $\low P(-X < -x) = 1-\up F_X =: \low{\hat F}_X$. Denote $U = \min\{ X, Y\} = -\max\{-X, -Y\}$.
	 Then we have,
	 \begin{align*}
	 	\low F_U(x) & = \low P(\min\{ X, Y\} \leqslant  x) \\
	 	 & = \low P(\max\{ -X,-Y\} \geqslant  -x)  \\
	 	 & = 1- \up P(\max\{ -X,-Y\} < -x) \\
	 	 & = 1-\up P(-X < -x) \up P(-Y < -x) \\
	 	 & = 1-\hat{\up F}_{X}(x) \hat{\up F}_{Y}(x) \\
	 	 & = 1-(1-\low F_X(x))(1-\low F_Y(x)) \\
	 	 & = \low F_X(x) + \low F_Y(x) - \low F_X(x) \low F_Y(x).
	 \end{align*}
This finishes the proof of (ii).
\end{proof}

\subsection{Application of imprecise Marshall's and maxmin copulas to shock models}
We now describe the shock model for the Marshall's case in the imprecise setting. Let $X$ and $Y$ be random variables, whose distributions are given in terms of $p$-boxes $(\low F_X, \up F_X)$ and $(\low F_Y, \up F_Y)$; and $Z$ a random variable with a precise distribution function $F_Z$. To every triple $(F_X, F_Y, F_Z)$ where $F_X \in \mathcal F_{(\low F_X, \up F_X)}$ and $F_Y\in \mathcal F_{(\low F_Y, \up F_Y)}$, there exist distribution functions $F, G$ and a Marshall's copula $C_{\marphi, \marpsi}$, so that $F$ and $G$ are the distributions of random variables $U = \max\{ X, Z\}$ and $V=\max\{ Y, Z \}$, and $C_{\marphi, \marpsi}(F, G)$ is their joint distribution function.  In particular, we will denote the minimal generating functions associated to the triple $(\low F_X, \low F_Y, F_Z)$ by $\low \marphi$ and $\low \marpsi$, as defined by \eqref{eq-low-phi}; and the corresponding maximal generating functions associated to the triple $(\up F_X, \up F_Y, F_Z)$ by $\up \marphi$ and $\up \marpsi$. Moreover, we will denote by $\low F$ and $\low G$ the distribution functions of $U$ and $V$ respectively corresponding to the triple $(\low F_X, \low F_Y, F_Z)$;  and by $\up F$ and $\up G$ the distribution functions of $U$ and $V$ respectively corresponding to the triple $(\up F_X, \up F_Y, F_Z)$.

\begin{thm}[Properties of imprecise Marshall's copulas]\label{main:marshall}
 In the situation described above we have:
	\begin{enumerate}[(i)]
		\item $\low \marphi \leqslant  \up \marphi$ and $\low{\marpsi}\leqslant  \up{\marpsi}$.
		\item $\low\marphi^*\circ\low F = \low \marpsi^* \circ\low G$ and $\up\marphi^*\circ\up F = \up \marpsi^* \circ\up G$.
		\item $\marcop_{\low{\marphi}, \low{\marpsi}}\leqslant \marcop_{{\marphi}, {\marpsi}}\leqslant \marcop_{\up{\marphi}, \up{\marpsi}}$, where $ \marcop_{{\marphi}, {\marpsi}}$ is the Marshall's copula corresponding to some triple $(F_X, F_Y, F_Z)$, where $F_X \in \mathcal F_{(\low F_X, \up F_X)}$ and $F_Y\in \mathcal F_{(\low F_Y, \up F_Y)}$.
		\item
		\[
			\low F = \low F_X F_Z \ \ \ \ \ \low G  = \low F_Y F_Z \ \ \ \ \
			\up F = \up F_X F_Z \ \ \ \ \  \up G  = \up F_Y F_Z.
		\]
		\item
		\begin{align*}
		\low F_X(x) & = \low{\marphi}(\low F(x)), \text{ if } \low F(x)>0; & \up F_X(x) & = \up{\marphi}(\up F(x)), \text{ if } \up F(x)>0; \\
		\low F_Y(y) & = \low{\marpsi}(\low G(y)), \text{ if } \low G(y)>0; & \up F_Y(y) & = \up{\marpsi}(\up G(y)), \text{ if } \up G(y)>0.
		\end{align*}
		\item $\low F\leqslant \up F$ and $\low G\leqslant \up G$.
		\item The distributions of the random variables $U = \max\{ X, Z \}$ and $V=\max\{ Y, Z\}$ are described with the $p$-boxes $(\low F, \up F)$ and  $(\low G, \up G)$  respectively.
		\item $\marcop_{\low{\marphi}, \low{\marpsi}}(\low F, \low G)\leqslant  \marcop_{\up{\marphi}, \up{\marpsi}}(\up F, \up G)$;
		\item The joint distribution of $(U, V)$ is described by a bivariate $p$-box
		\begin{equation*}
			(\low H, \up H) = (\marcop_{\low{\marphi}, \low{\marpsi}}(\low F, \low G), \marcop_{\up{\marphi}, \up{\marpsi}}(\up F, \up G)).
		\end{equation*}				
	\end{enumerate}	
\end{thm}
\begin{proof}
	(i) follows by Lemma~\ref{lem-phi-order}; (ii) follows directly from Proposition~\ref{prop-marshall-properties}; (iii) is a direct consequence of \eqref{eq-order-marcop} and (i); (iv) follows from Proposition~\ref{prop-marshall-properties}; (v) follows by definition; (vi) follows from Proposition~\ref{prop-order}; (vii) is a consequence of Proposition~\ref{prop-marshall-properties}; and (viii) follows from monotonicity of Marshall's copulas, (iii) and (vi).
	
	To prove (ix), let $(\low H, \up H)$ be the bivariate $p$-box describing the distribution of vector $(U, V)$. Using the factorization property we obtain:
	\begin{align}
		\low H(x, y) & = \low F_X(x) \low F_Y(y) F_Z(\min\{x, y \}) = \marcop_{\low{\marphi}, \low{\marpsi}}(\low F(x), \low G(y)); \label{eq-marlow}\\
		\up H(x, y) & = \up F_X(x) \up F_Y(y) F_Z(\min\{x, y \}) = \marcop_{\up{\marphi}, \up{\marpsi}}(\up F(x), \up G(y)).\label{eq-marup}
	\end{align}
\end{proof}
\begin{rem}
	Point (ix) of this theorem tells us that the ``first'' part of the imprecise version of Sklar's theorem does hold for the bivariate $p$-boxes representing shock models described by Marshall.
\end{rem}
\begin{ex}\label{ex-impr-marshall}
	Consider again Examples~\ref{ex-marcop} and \ref{ex-mmcop}, and suppose this time that we cannot assume precisely given parameters, but instead we consider the $p$-boxes $(\low F_X, \up F_X)$ and $(\low F_Y, \up F_Y)$, where $\low F(x)$ is an exponential distribution with parameter $\lambda$ and $\up F(x)$ with some parameter $\lambda' > \lambda$. It is immediate that $\low F < \up F$ holds. Similarly, let $\low F_Y$ and $\up F_Y$ be exponential with parameters $\eta < \eta'$ respectively. It is also easy to check that $\phi(u) \leqslant \phi'(u)$ where $\phi$ and $\phi'$ are given by \eqref{eq-ex-phi} and similarly, $\marpsi \leqslant \marpsi'$.
	
	So $\marcop_{\phi, \psi} \leqslant \marcop_{\phi', \psi'}$ and the bivariate distribution function of the vector $(U, V)$, where $U = \max\{ X, Z\}$ and $V=\max\{ Y, Z \}$ is then described with a bivariate $p$-box $(\low H, \up H)$, where the bounds are given by \eqref{eq-marlow} and \eqref{eq-marup}.
\end{ex}
We will now present the counterpart of our Marshall's model for the maxmin case. As before, let $(F_X, F_Y, F_Z)$ be a triple of distribution functions corresponding to independent random variables $X, Y$ and $Z$. We allow the distributions of $X$ and $Y$ to be given imprecisely in terms of $p$-boxes $(\low F_X, \up F_X)$ and $(\low F_Y, \up F_Y)$. We introduce the random variables $U = \max\{ X, Z\}$ and $W = \min\{ Y, Z \}$ and let $F$ and $K$ be their corresponding distribution functions. Furthermore, let $\mmcop_{\mmphi, \mmpsi}(F, K)$ denote the joint distribution function of the random vector $(U, W)$. In particular, let $\low \mmphi$ and $\low \mmpsi$ be the minimal functions associated to the triple $(\low F_X, \low F_Y, F_Z)$; and let $\up \mmphi$ and $\up \mmpsi$ be the maximal functions associated to the triple $(\up F_X, \up F_Y, F_Z)$. The distribution functions of random variables $U$ and $W$ corresponding to triples $(\low F_X, \low F_Y, F_Z)$ and $(\up F_X, \up F_Y, F_Z)$ will be denoted by $\low F, \low K$ and $\up F, \up K$ respectively.

\begin{thm}[Properties of imprecise maxmin copulas]\label{main:maxmin}
	In the above situation we have:
	\begin{enumerate}[(i)]
		\item $\low\mmphi\leqslant \up \mmphi$ and $\low \mmpsi\leqslant \up \mmpsi$;
		\item $\low \mmphi^* \circ \low F = \low \mmpsi_* \circ \low K$ and $\up \mmphi^* \circ \up F = \up \mmpsi_* \circ \up K$;
		\item $\mmcop_{\low{\mmphi}, \up{\mmpsi}} \leqslant \mmcop_{\mmphi, \mmpsi} \leqslant \mmcop_{\up{\mmphi}, \low{\mmpsi}}$ where $\marcop_{{\marphi}, {\marpsi}}$ is a maxmin copula corresponding to some triple $(F_X, F_Y, F_Z)$, where $F_X\in \mathcal F_{(\low F_X, \up F_X)}$ and $F_Y\in \mathcal F_{(\low F_Y, \up F_Y)}$;
		\item
		\begin{align*}
		\low F & = \low F_X F_Z & \low K & = \low F_Y + F_Z -\low F_Y F_Z \\
		\up F & = \up F_X F_Z & \up K & = \up F_Y + F_Z -\up F_Y F_Z
		\end{align*}
		\item
		\begin{align*}
		\low F_X(x) & = \low{\mmphi}(\low F(x)), \text{ if } \low F(x)>0; & \up F_X(x) & = \up{\mmphi}(\up F(x)), \text{ if } \up F(x)>0; \\
		\low F_Y(y) & = \low{\mmpsi}(\low K(y)), \text{ if } \low K(y)<1; & \up F_Y(y) & = \up{\mmpsi}(\up K(x)), \text{ if } \up K(y)<1.
		\end{align*}
		\item $\low F\leqslant \up F$ and $\low K\leqslant \up K$;
		\item The distributions of the random variables $U = \max\{ X, Z \}$ and $W=\min\{ Y, Z\}$ are described with the $p$-boxes $(\low F, \up F)$ and  $(\low K, \up K)$  respectively.
		\item $\mmcop_{\low{\mmphi}, \low{\mmpsi}}(\low F, \low K)\leqslant  \mmcop_{\up{\mmphi}, \up{\mmpsi}}(\up F, \up K)$;
		\item The joint distribution of $(U, V)$ is described by a bivariate $p$-box
		\begin{equation}\label{eq-maxmin-p-box}
		(\low H, \up H) = (\mmcop_{\low{\mmphi}, \low{\mmpsi}}(\low F, \low K), \mmcop_{\up{\mmphi}, \up{\mmpsi}}(\up F, \up K)).
		\end{equation}	
	\end{enumerate}
\end{thm}
\begin{proof}
	(i) follows from Lemmas~\ref{lem-phi-order} and \ref{lem-chi-order};	(ii) is a direct consequence of Proposition~\ref{prop-marshall-properties}; (iii) Follows from Proposition~\ref{prop-maxmin-bounds} and (i); (iv) and (v) follow from definitions and Proposition~\ref{prop-order}; (vi) follows from Proposition~\ref{prop-order} and together with Proposition~\ref{prop-maxmin-properties} implies (vii).
	
	Clearly, (ix) implies (viii). Therefore we prove (ix) directly. It has been shown in \cite{omladic2016} that the joint distribution function $H$ for $(U, V)$ has the form
	\begin{equation*}
	H(x, y) = \begin{cases}
	F_X(x)F_Z(x) & x\leqslant  y \\
	F_X(x)[F_Z(y)+F_Y(y)(F_Z(x)-F_Z(y))] & x\geqslant  y
	\end{cases}
	\end{equation*}
	The first part is clearly minimized by taking $F_X = \low F_X$ and in the second part, $x\geqslant  y$ implies that $F_Z(x)-F_Z(y)\geqslant  0$, and therefore this part is minimized by taking $\low F_X$ and $\low F_Y$ in place of $F_X$ and $F_Y$ respectively. Thus implying that $\low H(x, y) = \mmcop_{\low{\mmphi}, \low{\mmpsi}}(\low F, \low K)$. The proof for the upper bound is identical.
\end{proof}

\begin{ex}\label{ex-impr-maxmin}
	Suppose again that the distributions for $X$ and $Y$ are given in terms of $p$-boxes $(\low F_X, \up F_X)$ and $(\low F_Y, \up F_Y)$, where $\low F(x)$ is an exponential distribution with parameter $\lambda$ and $\up F(x)$ with some parameter $\lambda' > \lambda$, and similarly $\low F_Y$ and $\up F_Y$ exponential distributions with parameters $\eta < \eta'$ respectively. Additionally we we now have that $\mmpsi \leqslant \mmpsi'$. The joint distribution of the vector $(U, W)$ with components $U = \max\{ X, Z\}$ and $W=\min\{Y, Z\}$ is again given by $p$-box \eqref{eq-maxmin-p-box}.
\end{ex}
\begin{rem}\label{rem1}
	The last theorem shows that the bivariate distribution given by the bivariate $p$-box \eqref{eq-maxmin-p-box} is not the right one satisfying the imprecise version of Sklar's theorem. Indeed, if this were so, we would need to consider the boundary copulas applied to the boundary marginals
	\begin{equation*}
	(\mmcop_{\low{\mmphi}, \up{\mmpsi}}(\low F, \low K), \mmcop_{\up{\mmphi}, \low{\mmpsi}}(\up F, \up K)),
	\end{equation*}
	which however only gives outer bounds for the bivariate $p$-box \eqref{eq-maxmin-p-box}. According to Proposition~\ref{prop-maxmin-bounds}, the bounds are in general loose.
\end{rem}
\begin{rem}\label{rem2}
  Note, moreover, that inequality (viii) of the last proposition does not even imply that $\mmcop_{\low{\mmphi}, \low{\mmpsi}} \leqslant \mmcop_{\up{\mmphi}, \up{\mmpsi}}$. We have evidence of this fact provided by some examples, perhaps too computationally elaborate to be given here, and would also be somewhat off topic. So, although the pair $(\mmcop_{\low{\mmphi}, \low{\mmpsi}} ,\mmcop_{\up{\mmphi}, \up{\mmpsi}})$ is in general not an imprecise copula in the sense of \citet{montes2015}, it satisfies the same Equation \eqref{eq-maxmin-p-box} as the imprecise copula  existence of which is given by the imprecise version of the Sklar's theorem.
\end{rem}

\bibliographystyle{plain}

\section{Conclusion}

In this paper we presented a possible approach to shock model copulas in the imprecise setting. We have showed that if we want to keep a valid stochastic interpretation of the models involved we should not introduce the $p$-box kind of ordering on the level of copulas, it should be introduced in a deeper layer on the level of shocks. However, there is a hard to control interplay between the global and local shocks so that a clear interpretation of order can be followed only in one direction: either for local shocks or for global shocks. We decided to give here the development of the theory in which local shocks are imprecise and global shocks are precise. We believe that a theory of roughly equivalent complexity could be obtained if we studied the setting in which the global shocks are imprecise and the local ones are precise.

It might be quite interesting and definitely worth doing to treat both kind of shocks in an imprecise way simultaneously. It is not clear how to combine the orders on distributions of these two kinds of shocks when they come into an interplay during their action on components (cf.\ Subsection \ref{subsec:stochastic}). We leave this challenge for further study.

There are many more questions that are worth to be considered in the future investigations that this paper is opening. On one hand our imprecise versions of shock model induced copulas resulted in two constructions of imprecise copulas in the sense of Montes et al.\ \cite{montes2015} that satisfy an additional Condition (C) as seen in Corollaries \ref{cormar}\&\ref{cormm}. (Recall that this condition has recently been shown \cite{OmSt} truly additional.) On the other hand, according to Remarks \ref{rem1}\&\ref{rem2}, we found a pair of copulas that give some interesting imprecise information on the problem, i.e.\ Equation \eqref{eq-maxmin-p-box}, but does not satisfy the definition of an imprecise copula of Montes et al.\ \cite{montes2015} and neither it satisfies the version of the imprecise Sklar's theorem presented there. Definitely an observation that deserves further exploration,

\section*{Acknowledgements}
\begin{enumerate}
	\item The authors are grateful to prof.\ Bla\v{z} Moj\v{s}kerc for his help with figures of stochastic interpretation of shock model induced copulas.
	\item The authors are grateful to the two anonymous referees for careful readings of previous versions of this paper and for many valuable suggestions.
	\item Damjan \v Skulj acknowledges the financial support from the Slovenian Research Agency (research core funding No. P5-0168).
	\item Matja\v z Omladi\v c acknowledges the financial support from the Slovenian Research Agency (research core funding No. P1-0222).	
\end{enumerate}

\end{document}